\documentclass[11pt]{amsart}
\usepackage{a4,epsf,psfrag}

\usepackage{a4,amssymb,epsf,psfrag}
\usepackage{graphicx}
\usepackage{color}
\usepackage{amsfonts}
\usepackage{enumerate}
\usepackage{latexsym}
\usepackage{epsfig}
\usepackage{amsthm}
\usepackage{amsmath}
\usepackage{latexsym}
\usepackage[all]{xy}

\usepackage[
            breaklinks=true,
            colorlinks=true]{hyperref}
\usepackage[active]{srcltx}

\makeatletter\@addtoreset{equation}{section}\makeatother
\makeatletter\@addtoreset{figure}{section}\makeatother
\makeatletter\@addtoreset{table}{section}\makeatother

\makeatletter\@addtoreset{equation}{section}\makeatother
\makeatletter\@addtoreset{figure}{section}\makeatother
\makeatletter\@addtoreset{table}{section}\makeatother

\newtheorem{theorem}{Theorem}[section]

\newtheorem{theoremL}{Theorem}

\newcommand{\R}{{\mathbb R}}
\newcommand{\Z}{{\mathbb Z}}

\newcommand{\N}{{\mathbb N}}
\newcommand{\T}{{\mathbb T}}

\newcommand{\op}[1]{\!\!\mathop{\rm ~#1}\nolimits}

\newenvironment{remark}{\refstepcounter{theorem}\par\medskip\noindent{\bf
Remark~\thetheorem~~}}{\unskip\nobreak\hfill\hbox{ $\oslash$}\par\bigskip}

\newenvironment{example}{\refstepcounter{theorem}\par\medskip\noindent{\bf
Example~\thetheorem~~}}{\unskip\nobreak\hfill\hbox{ $\oslash$}\par\bigskip}

\newenvironment{definition}{\refstepcounter{theorem}\par\medskip\noindent{\bf
Definition~\thetheorem~~}}
\newcommand{\deriv}[2]{\frac{\partial #1}{\partial #2}}

\newcommand{\restr}{\upharpoonright}

\newcommand{\RM}{\mathbb{R}}

\newcommand{\NM}{\mathbb{N}}

\renewcommand{\geq}{\geqslant}
\renewcommand{\leq}{\leqslant}

\begin{document}

\title[The affine invariant of generalized semitoric systems]{The
  affine invariant of\\ generalized semitoric systems}

\medskip

\author{\'Alvaro Pelayo \,\,\,\,\,\,\, \,\,\,Tudor S. Ratiu
  \,\,\,\,\,\,\, \,\,\,San V\~{u} Ng\d{o}c}

\maketitle

\begin{abstract}
  A \emph{generalized semitoric system} $F:=(J,H)\colon M \to \R^2$ on
  a symplectic $4$\--manifold is an integrable system whose essential
  properties are that $F$ is a proper map, its set of regular values
  is connected, $J$ generates an $S^1$\--action and is not necessarily
  proper. These systems can exhibit focus\--focus singularities, which
  correspond to fibers of $F$ which are topologically multi\--pinched
  tori. The image $F(M)$ is a singular affine manifold which contains
  a distinguished set of isolated points in its interior: the
  focus\--focus values $\{(x_i,y_i)\}$ of $F$. By performing a
  vertical cutting procedure along the lines $\{x:=x_i\}$, we
  construct a homeomorphism $f \colon F(M) \to f(F(M))$, which
  restricts to an affine diffeomorphism away from these vertical
  lines, and generalizes a construction of V\~u Ng\d oc. The set
  $\Delta:=f(F(M)) \subset \R^2$ is a symplectic invariant of
  $(M,\omega,F)$, which encodes the affine structure of $F$. Moreover,
  $\Delta$ may be described as a countable union of planar regions of
  four distinct types, where each type is defined as the region
  bounded between the graphs of two functions with various properties
  (piecewise linear, continuous, convex, etc). If $F$ is a toric
  system, $\Delta$ is a convex polygon (as proven by Atiyah and
  Guillemin\--Sternberg) and $f$ is the identity.
\end{abstract}

\section{Introduction} \label{sec:intro}

Let $(M,\omega)$ be a symplectic $2n$\--manifold, that is, $M$ is a
smooth $2n$\--dimensional manifold $M$ and $\omega$ is a
non\--degenerate closed $2$\--form.  Throughout this paper we assume
that $M$ is connected. However, we do not assume that $M$ is compact.

\subsection{Definitions}

Motivated by \cite{At1982, GuSt1982,PeRaVN2011, PeVN2009, PeVN2010,
  VN2007}, we introduce in this paper a particular class of classical
Liouville integrable systems of the so-called ``generalized semitoric
type".

\begin{definition} \label{is} An \emph{integrable system} on
  $(M,\omega)$ is given by a map $F \colon M \to \R^n$ whose
  components $f_1,\ldots,f_n \colon M \to \mathbb{R}$ are Poisson
  commuting smooth functions which generate vector fields
  $\mathcal{X}_{f_1},\ldots,\mathcal{X}_{f_n}$ (via pairing with
  $\omega$) that are linearly independent at almost every point. A
  \emph{singularity} of $F$ is a point in $M$ where this linear
  independence fails to hold.  A \emph{singular fiber} of $F$ is a
  level set of $F$ which contains at least one singularity of $F$.
\end{definition}
 
\begin{figure}[h] \label{fig:cartographic2}.  \centering
  \includegraphics[height=5.5cm]{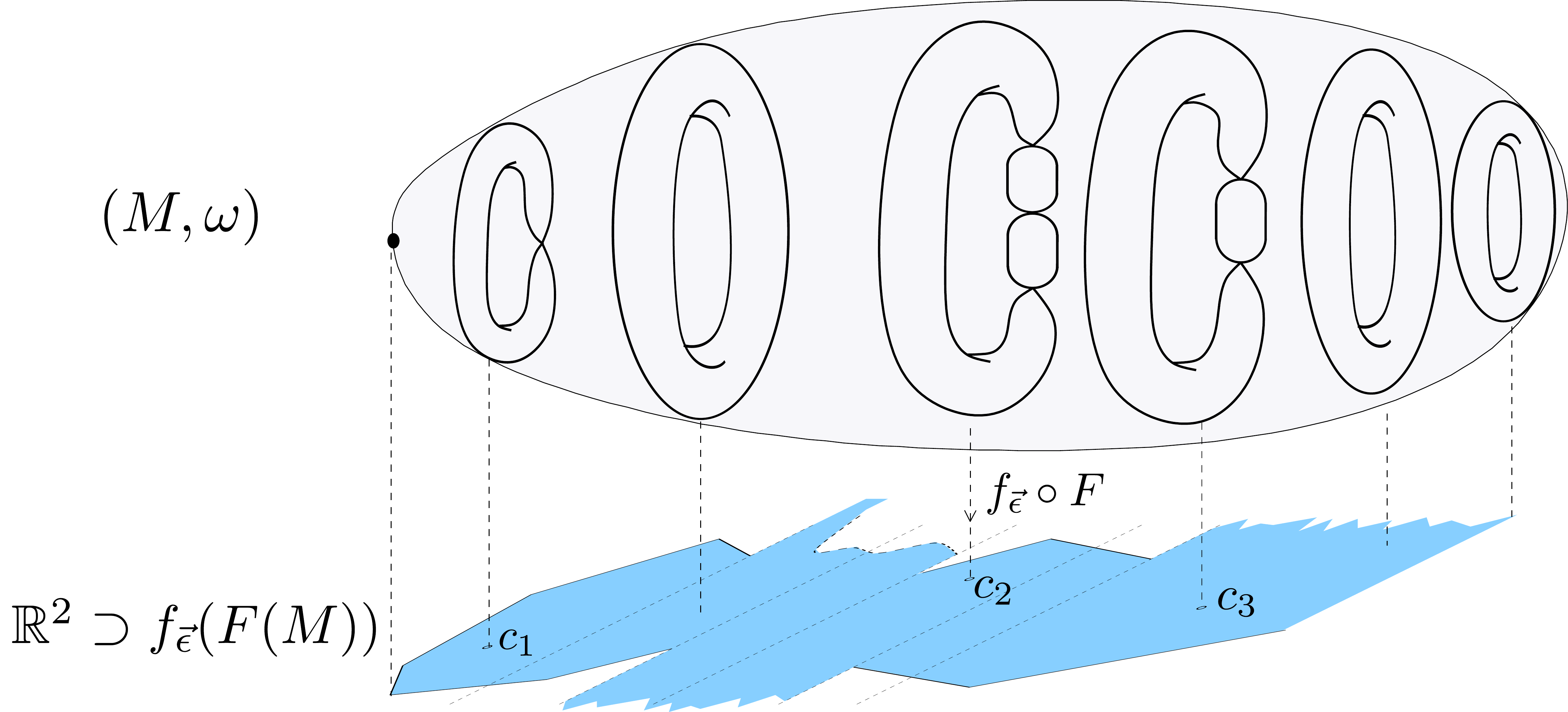}
  \caption{The singular Lagrangian fibration $F \colon M \to \R^2$ of
    a generalized semitoric system with three isolated singular values
    $c_1,\,c_2,\,c_3$.  The generic fiber is a $2$\--dimensional torus,
    the singular fibers are lower dimensional tori, points, or
    multipinched tori. For each $\vec \epsilon \in \{-1,1\}^2$ we
    construct, in Theorems~\ref{theo:polygon} and \ref{theo:b}, a
    homeomorphism $f_{\vec \epsilon} \colon F(M) \to \R^2$ such that
    $(f_{\vec \epsilon} \circ F)(M)$ is a ``nice region" of $\R^2$,
    which is a symplectic invariant. The notion of ``nice region" is
    made precise in Definition~\ref{def:type}.}
\end{figure}

In this article we assume that $n=2$, and use the index free notation
$f_1=J$ and $f_2=H$.

\begin{definition}
  An $S^1$-action on $(M,\omega)$ is \emph{Hamiltonian} if there
  exists a smooth map $J \colon M \to \mathbb{R}$, the \emph{momentum
    map}, such that
$$\omega(\mathcal{X}_M, \cdot) = - {\rm d} J,$$
where $\mathcal{X}_M$ is the infinitesimal generator of the action.
\end{definition}

In this article we construct a symplectic invariant when $F$ is of
\emph{generalized semitoric type}. We refer to
Section~\ref{sec:Eliasson} for a quick review of the notions
concerning singularities used in the following definition.

\begin{definition} \label{gst} An integrable system $F:=(J,H) \colon M
  \to \R^2$ on $(M,\omega)$ is \emph{generalized semitoric} if:
  \begin{enumerate}[({\rm H}.i)]
  \item\label{Hii} $J$ is the momentum map of an effective Hamiltonian
    circle action.
  \item \label{Hiii} The singularities of $F$ are non-degenerate with
    no hyperbolic blocks.
  \item \label{H-proper} $F$ is a proper map (i.e., the preimages of
    compact sets are compact).
  \item \label{H-connected} $J$ has connected fibers, and the
    bifurcation set of $J$ is discrete (here discrete includes
    multiplicity: that is, for any critical value $x$ of $J$, there
    exists a small neighborhood $V\ni x$ such that the critical set of
    $J$ in the preimage $J^{-1}(V)$ only contains a finite number of
    connected components.)
  \end{enumerate}
\end{definition}

\begin{remark}
  (H.\ref{H-connected}) implies that the fibers of $F$ are also
  connected by \cite{PeRaVN2011}. (H.\ref{H-proper}),
  (H.\ref{H-connected}) are implied by (H.\ref{Hii}),(H.\ref{Hiii})
  when $J$ is proper. In some simple physical models like the
  spherical pendulum (Example~\ref{sph}), $J$ is not proper but
  (H.\ref{H-proper}), (H.\ref{H-connected}) still hold.
\end{remark}

A typical generalized semitoric system is depicted in
Figure~\ref{fig:cartographic2}.
 
For background material on integrable systems and group actions, see
\cite{PeVN2011}.

\subsection{Singularities}
The class of systems in Definition \ref{gst} may have the so called
\emph{focus\--focus singularities}, and give rise to fibers of $F$
which are multi\--pinched tori. Focus\--focus singularities appear in
algebraic geometry \cite{GS06} and symplectic topology, e.g.,
\cite{LeSy2010, Sy2001, Vi2013} (in the context of Lefschetz
fibrations they are sometimes called \emph{nodes}), and include simple
physical models from mechanics such as the spherical pendulum
(\cite{AbMa1978}).

 \begin{figure}[h]
   \centering
   \begin{minipage}[c]{0.45\linewidth}
     \includegraphics[width=\textwidth]{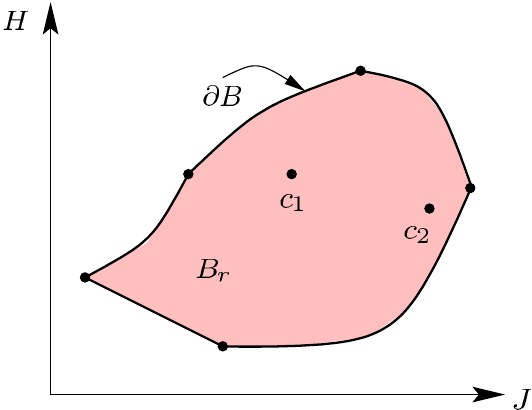}
   \end{minipage}\hfill%
   \begin{minipage}[c]{0.45\linewidth}
     \includegraphics[width=\textwidth]{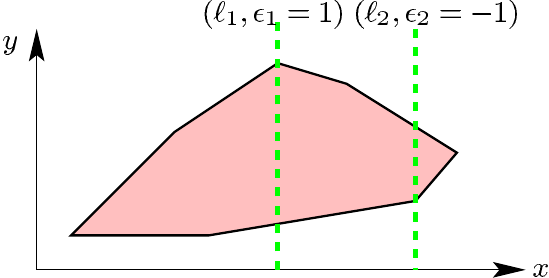}
   \end{minipage} \caption{In general $F(M) \subseteq \R^2$ is not
     convex. The interior of $F(M)$ contains two isolated singular
     values $c_1=(x_1,y_1)$ and $c_2=(x_2,y_2)$. By performing a
     vertical cutting procedure along the lines $\ell_i:=\{x:=x_i\}$,
     we construct a homeomorphism $f \colon F(M) \to f(F(M))$, which
     restricts to an affine diffeomorphism away from these vertical
     lines. The right hand side figure displays the associated polygon
     with the distinguished lines.  }
   \label{xx}
 \end{figure}

\begin{figure}[h]
  \centering
  \includegraphics[height=5.5cm]{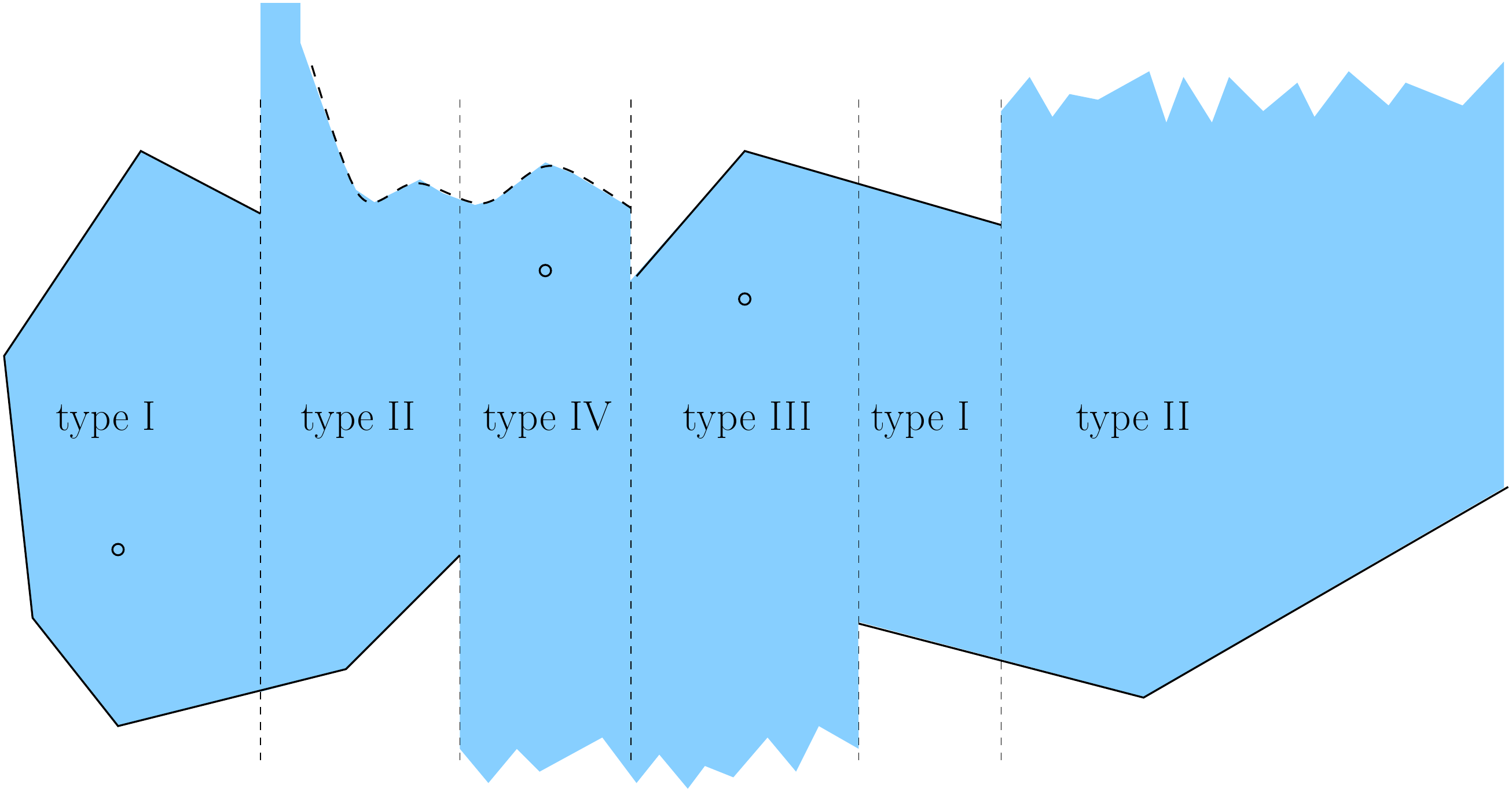}
  \caption{A cartographic projection of $F$. It is a symplectic
    invariant of $F$, see Theorem~\ref{theo:b}.}
  \label{fig:cartographic}
\end{figure}

\subsection{Toric systems: Atiyah and Guillemin--Sternberg
  Theory} \label{XX}

If $M$ is compact and $F:=(J,H)$ is the momentum map of an effective
Hamiltonian $2$\--torus action, all assumptions above hold by the
Atiyah and Guillemin\--Sternberg Theorem (\cite{At1982},
\cite{GuSt1982}) and $F$ does not possess any focus\--focus
singularity.  In this case $(M,\omega,F)$ is called a \emph{toric
  system} or a \emph{symplectic toric manifold}.

Toric systems have been thoroughly studied in the past thirty years
(in any dimension) and, at least from the point of view of symplectic
geometry, a complete picture emerged in the compact case due to the
aforementioned results of Atiyah \cite{At1982}, Guillemin-Sternberg
\cite{GuSt1982}, and a classification result due to Delzant
\cite{Delzant1988}. The first two papers showed that the image
$\mu(M)$ of the momentum map $\mu \colon M \to \R^k$ of a Hamiltonian
$\mathbb{T}^k$-action on a $2n$-dimensional compact connected
symplectic manifold $(M,\omega)$ is a convex polytope in $\R^k$, which
is a symplectic invariant.

\subsection{Goal of this article}

In the present article we will extend to generalized semitoric systems
the results in Section~\ref{XX}, inspired by an extension of the
Atiyah\--Guillemin\--Sternberg result to (non-generalized) semitoric
systems recently achieved by V\~u Ng\d oc \cite{VN2007}. Using Morse
theory and the Duistermaat\--Heckman Theorem for proper momentum maps,
he dealt with integrable systems $F \colon M \to \R^2$ of
\emph{semitoric} type for which, in addition to assumptions
(H.\ref{Hii})\--(H.\ref{H-connected}), $J\colon M \to \R$ is proper.
Then he performed a cutting procedure along the vertical lines going
through the isolated singularities of the image $F(M)$ of the system
and constructed a convex polygon from it, which is an invariant of
$F$; see Figure~\ref{xx}.

The difficulty of the generalized situation considered in this article
is due to the fact that the Duistermaat\--Heckman theorem does not
hold for nonproper $J$ (Remark~\ref{DH}), and neither does standard
Morse theory. This has striking consequences for the statement of our
extension: while the invariant in \cite{VN2007} is a class of convex
polygons as in Figure~\ref{xx}, ours is a union of planar regions of
various types (to be precisely defined later), which looks, in
general, like Figure~\ref{fig:cartographic}. This invariant encodes the singular affine structure
induced by the (singular) Lagrangian fibration $F \colon M \to \R^2$ on the
base $F(M)$.  Its construction and properties appear in 
Theorems~\ref{theo:polygon}, \ref{theo:b},
\ref{theo:c}. This affine structure also plays a role in
parts of symplectic topology, mirror symmetry, and algebraic
geometry, see for instance  Auroux~\cite{Au09}, Borman\--Li\--Wu~\cite{BLW13}, 
Kontsevich\--Soibelman~\cite{KS}. Integrable systems exhibiting semitoric 
features appear  in the theory of symplectic quasi-states, see
Eliashberg\--Polterovich \cite{eliashberg}.

 Theorem~\ref{niceexample}
shows that there are many simple examples in which the invariant,
which is the most natural planar representation of the singular affine
structure of the system, has a non\--polygonal, non-convex, form.

\section{Toric and semitoric systems}
\label{sec:toric-semitoric}
Although this section is not original, we put previous results in a
general framework which is better suited for expressing our new
results in the following section.

\subsection{The set of semitoric images}
\label{sec:groupT}
Let $\mathcal{P}(\R^2)$ be the set of subsets of $\R^2$
and $$\mathfrak{F}:=\Big\{\Z,\Z^{+},\,\Z^{-},\,
\{1,\ldots,N\}_{N>0},\,\varnothing\Big\},$$ where $\Z^+=\{n\in\Z\mid
n\geq 0\}$ and $\Z^-=\{n\in\Z\mid n\leq 0\}$.

Let
$$
T:=
\begin{bmatrix}
  1&0\\
  1&1
\end{bmatrix},
$$ 
and consider the group $\mathcal{T}$ whose elements are the matrices
$T^k$, $k \in \Z$, composed with vertical translations.  This gives
rise to the quotient space
$\mathcal{P}^{\mathcal{T}}(\R^2):=\mathcal{P}(\R^2)/
\mathcal{T}$.

\subsection{Action on
  $\mathcal{P}^{\mathcal{T}}(\R^2)\times\R^Z\times\N^Z$}

A vertical line $\mathcal{L} \subset \mathbb{R}^2$ splits
$\mathbb{R}^2$ into two half-spaces. Let $u \in \Z$. We define a map
$t^{u}_{\mathcal{L}}$ acting on $\R^2$ as follows. On the left half
space defined by $\mathcal{L}$, we let the map $t^{u}_{\mathcal{L}}$
act as the identity.  On the right half space, with an origin placed
arbitrarily on $\mathcal{L}$, $t^{u}_{\mathcal{L}}$ acts as the matrix
$T^{u}$.

\begin{definition}
  Let $Z \in \mathfrak{F}$ and let $\vec{x} \in \R^Z$. Let $n\in Z$
  and denote by $\mathcal{L}_n^{\vec{x}}$ the vertical line through
  $(\vec{x}(i),0)$.  We define the action of $\vec u \in \Z^{Z}$ on
  $\mathcal{P}^{\mathcal{T}}(\R^2)\times\R^Z$ by
$$
\vec{u} \cdot (X,\,\vec{x})=\left( \left(\prod_{\vec u(i) \neq 0,
      \,\,\, n \in Z} t^{\vec
      u(i)}_{\mathcal{L}_n^{\vec{x}}}\right)(X), \,\,\vec{x} \right).
$$
Let $\vec k\in\NM^Z$. We finally define the action of $\vec\epsilon\in
\{-1,1\}^Z$ on $\mathcal{P}^{\mathcal{T}}(\R^2)\times\R^Z\times\N^Z$
by the formula
\[
\vec\epsilon \cdot (X,\,\vec{x},\,\vec{k}) = \left((\vec\epsilon \cdot
  \vec k) \cdot (X,\,\vec{x}), \vec k\right),
\]
where $\vec\epsilon \cdot \vec k:= i\mapsto
\frac{1-\epsilon(i)}2k(i)$.  We denote the $\{-1,1\}^{Z}$\--orbit
space by $ \mathfrak{B}_{{\rm
    GST}}(Z):=\mathcal{P}^{\mathcal{T}}(\R^2)\times\R^Z\times\N^Z/{\{-1,1\}^{Z}}.
$
\end{definition}

\subsection{Affine invariant for semitoric systems}

Let $F=(J,H) \colon M \to \R^2$ be a semitoric system, i.e., in
addition to assumptions (H.\ref{Hii})\--(H.\ref{H-connected}), the map
$J\colon M \to \R$ is proper.  There exists a unique $Z \in
\mathfrak{F}$ such that $\vec x \in \R^{Z}$ is the tuple of images by
$J$ of focus\--focus values $c_i=(x_i,y_i)$ of $F$ ordered by
non\--decreasing values, and $\vec k\in\N^Z$ such that $\vec k(i)$ is
the number of focus-focus critical points in the fiber $F^{-1}(c_i)$.

Let $\mathcal{L}^{\vec x}:=(\mathcal{L}_{x_i})_{i \in Z}$ where
$\mathcal{L}_{x_i}$ is the unique vertical line in $\R^2$ through
$(x_i,0)$.  For each fixed $\vec \epsilon \in \{-1,1\}^Z$, V\~u Ng\d
oc constructed \cite[Theorem~3.8 and Proposition~4.1]{VN2007} an
equivalence class of convex polygons in $\R^2$
\begin{eqnarray} \label{xxx} (\Delta_{\vec \epsilon}\,\,\, {\rm mod}
  \,\,\mathcal{T}) \in \mathcal{P}^{\mathcal{T}}(\R^2).
\end{eqnarray}
by performing a cutting procedure along the vertical lines
$\mathcal{L}_{x_i}$.  The ``choice" of cuts is given by
$\vec\epsilon$, where a positive sign corresponds to an upward cut,
and a negative sign corresponds to a downward cut.

\begin{definition}
  Let $(M,\omega, F)$ be a semitoric system.  Define:
  \begin{eqnarray} \label{xxx2} \Delta(M,\omega, F):= (\Delta_{\vec
      \epsilon}\,\,\, {\rm mod}\,\, \mathcal{T},\,\vec{x}, \,\vec k)
    \mod \{-1,1\}^Z \in \mathfrak{B}_{{\rm GST}},
  \end{eqnarray}
  where $\vec\epsilon(i)=1$ for all $i\in Z$ and the action of $
  \{-1,1\}^Z$ is defined above.
\end{definition}

\begin{definition} 
  Let $\mathcal{M}_{{\rm T}}$ be the set of toric systems. Let $Z \in
  \mathfrak{F}$ and let $\mathcal{M}_{{\rm ST}}(Z)$ and
  $\mathcal{M}_{{\rm GST}}(Z)$ be the sets of semitoric and
  generalized semitoric systems $F=(J,H)$, with the images of the
  focus\--focus values of $F$ by $J$ indexed by the set $Z$. Let
  \begin{eqnarray} {\mathcal M}_{{\rm ST}}:= \bigsqcup_{Z \in
      \mathfrak{F}} {\mathcal M}_{{\rm ST}}(Z); \,\,\,\,\,\,\,\,\,\,\,
    {\mathcal M}_{{\rm GST}}:= \bigsqcup_{Z \in \mathfrak{F}}
    {\mathcal M}_{{\rm GST}}(Z); \,\,\,\,\,\,\,\,\,\,\,
    \mathfrak{B}_{{\rm GST}}:= \bigsqcup_{Z \in \mathfrak{F}}
    \mathfrak{B}_{{\rm GST}}(Z). \nonumber \nonumber
  \end{eqnarray}
\end{definition}

We now recall the notion of isomorphism for generalized semitoric
systems, which coincide with the notion introduced in~\cite{VN2007}
for proper semitoric systems.

\begin{definition} \label{def:iso} The generalized semitoric systems
  $(M_1,\, \omega_1,\,F_1:=(J_1,\,H_1))$ and $(M_2,\, \omega_2,\,
  F_2:=(J_2,\,H_2))$ are \emph{isomorphic} if there exists a
  symplectomorphism $\varphi: M_1 \to M_2$ such that
  $\varphi^*(J_2,\,H_2)=(J_1,\,h(J_1,\,H_1))$ for a smooth $h$ such
  that $\frac{\partial h}{\partial H_1}>0$.
\end{definition}
 
Notice that the set $\mathcal{M}_\textup{T}$ is not invariant under
these isomorphisms. Hence we introduce the following definition.

\begin{definition}
  A generalized semitoric system is said to be of \emph{toric type} if
  it is isomorphic to a toric system. We denote by
  $\mathcal{M}_{\textup{TT}}$ the set of semitoric systems of toric
  type.
\end{definition}

 \begin{remark}
   Clearly $ \mathcal{M}_{{\rm T}}\,\subsetneq\, {\mathcal M}_{{\rm
       ST}}\, \subsetneq \, \mathcal{M}_{{\rm GST}}\,.  $
 \end{remark}

 \begin{theorem}[\cite{VN2007}]
   The class of convex polygons {\rm (\ref{xxx2})} is an invariant of
   the isomorphism type of $F$.\footnote{while $F(M)$ is neither
     generally convex, nor an invariant.}
 \end{theorem}

 For a system satisfying properties
 (H\ref{Hii})--(H\ref{H-connected}), in this article we will construct
 a more general symplectic invariant by unwinding the (singular)
 affine structure induced by $F$ on $F(M)$, which extends (\ref{xxx}).
 The fact that $J$ may not be proper complicates the situation a lot
 because the Duistermaat\--Heckman theorem does not hold for nonproper
 momentum maps (Remark~\ref{DH}), and standard Morse theory
 essentially breaks down for nonproper maps.  This is why systems with
 non-proper $J$ were excluded in \cite{PeVN2009,PeVN2010}.

 \section{Summary result: Theorem~\ref{stu}}
 \label{sec:summary}
 As a consequence of Theorems~\ref{theo:polygon}, \ref{theo:b}, and
 \ref{theo:c} (stated and proved in the next sections), we obtain the
 following statement, which is less explicit (less useful for
 computations) but provides a summary of the paper.

  \begin{definition}
    Recall that if $(M,\omega,F) \in \mathcal{M}_{{\rm T}}$ then
    $F(M)$ does not contain focus\--focus singular values, and $F(M)$
    is a convex polygon.  If $(M,\omega,F) \in \mathcal{M}_{{\rm
        ST}}$, let $\Delta(M,\omega,F)$ be as in (\ref{xxx2}).
    Consider the maps
    \begin{eqnarray} \label{p1} \mathcal{C}_{\rm ST} : {\mathcal
        M}_{{\rm ST}} \ni (M,\omega,F) \longmapsto \Delta(M,\omega,F)
      \in \mathfrak{B}_{{\rm GST}}
    \end{eqnarray}
    \begin{equation} \label{p2} \mathcal{C}_{\rm T} : {\mathcal
        M}_{{\rm T}} \ni (M,\omega,F)\longmapsto (F(M) \,\,\, {\rm
        mod}\,\, \mathcal{T}, \,\varnothing, \varnothing) \mod
      \{-1,1\}^Z \in \mathfrak{B}_{{\rm GST}}
    \end{equation}
    and
    \begin{equation} \label{p2} \mathcal{C}_{\rm TT} : {\mathcal
        M}_{{\rm TT}} \ni (M,\omega,F)\longmapsto (F'(M) \,\,\, {\rm
        mod}\,\, \mathcal{T}, \,\varnothing, \varnothing) \mod
      \{-1,1\}^Z \in \mathfrak{B}_{{\rm GST}},
    \end{equation}
    where $F'$ is any toric momentum map isomorphic to $F$ as a
    semitoric system.
  \end{definition}
 
 \begin{definition}
   If $\mathcal{F}$ is a family of integrable systems containing
   $\mathcal{M}_{{\rm T}}$, a \emph{cartographic invariant} is any map
   $\mathcal{C} \colon \mathcal{F} \to \mathfrak{B}_{{\rm GST}}$
   extending $\mathcal{C}_{{\rm T}}$ in (\ref{p2}) and invariant under
   isomorphism.
 \end{definition}
 It follows from the Atiyah\--Guillemin\--Sternberg theory
 and~\cite{VN2007} that the maps $\mathcal{C}_{\textup{T}}$,
 $\mathcal{C}_{\textup{TT}}$ and $\mathcal{C}_{\textup{ST}}$ are
 cartographic invariants.  Notice that it is straightforward to check
 from Definition~\ref{def:iso} that $\mathcal{C}_\textup{TT}$ is
 indeed well defined.

 \begin{theoremL} \label{stu} Let $\mathcal{C}_{{\rm ST}}$,
   $\mathcal{C}_{{\rm TT}}$ and $\mathcal{C}_{{\rm T}}$ be the
   cartographic invariants defined in {\rm (\ref{p1})} and {\rm
     (\ref{p2})}.  Then there exists a cartographic invariant $
   \mathcal{C}_{{\rm GST}} \colon {\mathcal M}_{{\rm GST}}\to
   \mathfrak{B}_{{\rm GST}}$ such that the diagram
   \begin{equation}\label{mainlemma}
     \xymatrix{
       {\mathcal M}_{{\rm T} }  \ar@{^{(}->}[r] \ar_{\mathcal{C}_{\rm
           T}}[rrrd] & {\mathcal M}_{\rm TT}
       \ar@{^{(}->}[r]  \ar^{\mathcal{C}_{\rm TT}}[rrd] & 
       {\mathcal M}_{\rm ST} \ar@{^{(}->}[r]  \ar^{\mathcal{C}_{\rm
           ST}}[rd] & {\mathcal
         M}_{\rm GST} \ar^{\mathcal{C}_{\rm GST}}[d]\\
       & & & \mathfrak{B}_{\rm GST}
     }
   \end{equation}
   is commutative.
 \end{theoremL}

 We will prove several theorems which together imply Theorem~\ref{stu}
 and which are more informative because the cartographic invariant is
 explicitly constructed. It would be interesting to prove Theorem
 \ref{stu} (in particular, defining the maps involved) for integrable
 systems on origami manifolds (see~\cite{DGP}) and on orbifolds
 (see~\cite{LT}), where, as far as we know, integrable systems have
 not been studied.

 \section{Main results: Theorems~\ref{theo:polygon}, \ref{theo:b},
   \ref{theo:c}, \ref{niceexample}} \label{sec:main}

 For simplicity, from now on, we use the term ``semitoric" to refer to
 integrable systems satisfying {\rm
   (H\ref{Hii})--(H\ref{H-connected})}, that is, we drop the word
 ``generalized".

 Let $(M, \omega)$ be a connected symplectic $4$-manifold and
 $F:=(J,H) \colon M \to \R^2$ a semitoric system. Next we prepare the
 grounds for the main theorems of the paper.  Let $B_r \subset B$ is
 the set of regular values of $F$.  Since $F$ is proper we know that
 the set of focus-focus critical values of $F$ is discrete.  We denote
 by $c_i:=(x_i,y_i)$, $i \in Z$, the focus-focus critical values of
 $F$, ordered so that $x_i\leq x_{i+1}$, and $k_i$ is the number of
 critical points in $F^{-1}(c_i)$.  Given
 $\vec\epsilon=(\epsilon_i)_{i\in Z}\in\{-1,+1\}^{Z}$, we define the
 vertical closed half line originating at $c_i=(x_i,y_i)$
 by $$\mathcal{L}_i^{\epsilon_i}:= \{(x_i,y) \in \mathbb{R}^2 \mid
 \epsilon_i y\geq\epsilon_i y_i\}$$ for each $i \in Z$, which is
 pointing up from $c_i$ if $\epsilon_i = 1$ and down if $\epsilon_i =
 -1$.  Define $\ell^{\epsilon_i}_i:=B\cap \mathcal{L}_i^{\epsilon_i}
 \subset \mathbb{R}^2$.

 For any $c \in B$, define $I_c:= \left\{i \in Z\mid c \in
   \ell^{\epsilon_i}_i \right\}$ and the map
 $k:\mathbb{R}^2\rightarrow \mathbb{Z}$ by
 \begin{equation}
   k(c): = \sum_{i \in I_c} \epsilon_ik_i\,,
   \label{equ:k}
 \end{equation}
 with the convention that if $I_c = \varnothing$ then $k(c) =0$. The
 sum is finite thanks to (H.\ref{H-connected}). Let
 $\ell^{\vec{\epsilon}}: = k ^{-1}(\mathbb{Z}\setminus \{0\})$.

 For the necessary background on affine manifolds in the discussion
 which follows, readers may consult the appendix
 (Section~\ref{sec:affine}).  We write $\mathbb{A}^2_{\mathbb{Z}}$ for
 $\mathbb{R}^2$ equipped with its standard integral affine structure
 with automorphism group $\operatorname{Aff}(2,\mathbb{Z}):=
 \operatorname{GL}(2,\mathbb{Z})\ltimes\mathbb{R}^2$.

 The integral affine structure on $B_r$, which in general is not the
 affine structure induced by $\mathbb{A}^2_{\mathbb{Z}}$, is defined
 for instance in \cite[Section~3]{VN2007} or \cite[Appendix
 A2]{HoZe1994}; see also Section~\ref{sec:affine}: affine charts near
 regular values are given by action variables $f \colon U \subset F(M)
 \to \R^2$ on open subsets $U$ of $F(M)$ with the induced subspace
 topology and any two such charts differ by the action of ${\rm
   Aff}(2,\Z)$.

 Let $X$ and $Y$ be smooth manifolds and $A\subset X$. A map $f:A\to
 Y$ is said to be \emph{smooth} if every point in $A$ admits an open
 neighborhood in $X$ on which $f$ can be smoothly extended. The map
 $f$ is called a \emph{diffeomorphism onto its image} if $f$ is
 injective, smooth, and its inverse $f^{-1} \colon f(A) \to A$ is a
 smooth map, in the sense above.

 The following theorem is a generalization of~\cite[Theorem
 3.8]{VN2007}.
 \begin{theoremL}
   \label{theo:polygon}
   Let $F \colon M \to \mathbb{R}^2$ be a semitoric system in
   $\mathcal{M}_\textup{GST}(Z)$, for some $Z\in\mathfrak{F}$.  For
   every $\vec\epsilon\in\{-1,+1\}^Z$ there exists a
   homeomorphism $$f_{\vec \epsilon} \colon B \to f_{\vec \epsilon}
   (B)\subseteq \mathbb{R}^2$$ of the form $f_{\vec \epsilon}
   (x,y)=(x,f^{(2)}_{\vec\epsilon}(x,y))$ such that:
   \begin{enumerate}[{\rm (P.i)}]
   \item \label{pi} the restriction
     $f_{\vec{\epsilon}}|_{\left(B\setminus
         \ell^{\vec{\epsilon}}\right)}$ is a diffeomorphism onto its
     image, with positive Jacobian determinant;
   \item \label{pii} the restriction
     $f_{\vec\epsilon}|_{\left(B_r\setminus\ell^{\vec\epsilon}
       \right)}$ sends the integral affine structure of $B_r$ to the
     standard integral affine structure of
     $\mathbb{A}^2_{\mathbb{Z}}$;
   \item \label{piii} the restriction $f_{\vec
       \epsilon}|_{\left(B_r\setminus\ell^{\vec\epsilon} \right)}$
     extends to a smooth multi-valued map $B_r \to \mathbb{R}^2$ and
     for any $i \in Z$ and $c\in\ell_i^{\epsilon_i}\setminus \{c_i\}$,
     we have
     \begin{eqnarray} \label{jump} \lim_{\substack{(x,y)\rightarrow
           c\\x<x_i}} \op{d}\!f_{\vec \epsilon} (x,y) =
       T^{k(c)}\lim_{\substack{(x,y)\rightarrow
           c\\x>x_i}}\op{d}\!f_{\vec \epsilon} (x,y),
     \end{eqnarray}
     where $k(c)$ is defined in~\eqref{equ:k}.
   \end{enumerate}
   Such an $f_{\vec \epsilon}$ is unique modulo a left composition by
   a transformation in $\mathcal{T}$.
 \end{theoremL}

 In toric case, $f_{\vec{\epsilon}} (x,y) = (x,y)$, as was mentioned
 in Section~\ref{sec:summary}.

  \begin{definition}
    The map $f_{\vec \epsilon}$ in Theorem \ref{theo:polygon} is a
    \emph{cartographic map\footnote{since they lay out the affine
        structure of $F$ on two dimensions.}  for $F$} and its image
    $f_{\vec \epsilon}(B)$ is a \emph{cartographic projection of $F$}.
  \end{definition}

  \begin{definition}
    \label{def:type}
    Let $\mathcal{R}$ be a subset of $\mathbb{R}^2$. We say that
    $\mathcal{R}$ has \emph{type I} if there is a convex polygon
    $\Delta \subset \R^2$ and an interval $I \subseteq \R$ such that
$$\mathcal{R}=\Delta \cap \Big\{(x,y) \in \R^2\,\,|\, \, x \in I\Big\}.$$
We say that $R$ has \emph{type II} if there is an interval $I
\subseteq \R$ and $f\colon I \to \R$, $g \colon I \to \overline{\R}$
such that $f$ is piecewise linear, continuous, and convex, $g$ is
lower semicontinuous, and
$$\mathcal{R}=\Big\{(x,y) \in \R^2\,\,|\, \, x \in I\,\,\,\, {\rm and}\,\,\,\, f(x) \leq y < g(x)\Big\}.$$
We say that $\mathcal{R}$ has \emph{type III} if there is an interval
$I \subseteq \R$ and $f \colon I \to \overline{\R}$, $g \colon I \to
\R$ such that $f$ is upper semicontinuous, $g$ is piecewise linear
continuous and concave, and
$$\mathcal{R}=\Big\{(x,y) \in \R^2\,\,|\, \, x \in I \,\,\,\, {\rm and}\,\,\,\,  f(x) < y \leq g(x)\Big\}.$$
We say that $R$ has \emph{type IV} if there is an interval $I
\subseteq \R$ and $f,g \colon I \to \overline{\R}$ such that $f$ is
upper semicontinuous, $g$ is lower semicontinuous, and
$$\mathcal{R}=\Big\{(x,y) \in \R^2\,\,|\, \, x \in I \,\,\,\, {\rm and}\,\,\,\,  f(x) < y< g(x)\Big\}.$$
\end{definition}

In the following statement we call a \emph{discrete} sequence a
sequence such that for every value $c$, there is a neighborhood of $c$
which contains only the image of a finite number of indices.
\begin{theoremL} \label{theo:b} Let $F=(J,H) \colon M \to
  \mathbb{R}^2$ be a semitoric system and let $f_{\vec \epsilon}$ be a
  cartographic map for $F$.  Let
$$
K^{+}:=\Big\{x \in J(M) \mid \,\,\,\, J^{-1}(x) \cap
H^{-1}([0,+\infty)) \,\, {\rm is \,\, compact} \Big\}.
$$
and
$$
K^{-}:=\Big\{x \in J(M) \mid \,\,\,\, J^{-1}(x) \cap
H^{-1}((-\infty,0]) \,\, {\rm is \,\, compact} \Big\}.
$$
Suppose that the topological boundaries $\partial K^{+}$ and $\partial
K^{-}$ in $J(M)$ are discrete.  Then there exists an increasing
sequence $(x_j)_{j \in \Z}$ in $\mathbb{R}$, and sets
$\mathcal{C}^{\vec \epsilon}_j \subset \R^2$, $j \in \Z$, such that:
\begin{enumerate}[{\rm (P.1)}]
\item \label{bp1} for each $j \in \Z$, the set $\mathcal{C}^{\vec
    \epsilon}_j$ has type I, II, III, or IV associated to
  $(x_j,x_{j+1})$;
\item \label{bp2} $f_{\vec \epsilon} (B)=\bigcup_{j \in \Z}
  \mathcal{C}^{\vec \epsilon}_j$;
\item \label{bp4} for every $j \in \Z$, and every regular value $x$ of
  $J$, the volume $V(x)\leq +\infty$ of $J^{-1}(x)$ is equal to the
  Euclidean length of the vertical line segment $(\{x\} \times \R)
  \cap \mathcal{C}^{\vec \epsilon}_j$.
\end{enumerate}
\end{theoremL}
In some cases (for instance if $M$ is compact), only a finite number
of the $x_j$'s are relevant.

Suppose that $F: M \rightarrow \mathbb{R}^2$ is the momentum map of a
Hamiltonian $\mathbb{T}^2$-action on a compact connected symplectic
4-manifold. Then the cartographic projection of $F$ is a compact
convex polygon in $\mathbb{R}^2$; see \cite{At1982} and
\cite{GuSt1982}. If $F:M \rightarrow \mathbb{R}^2$ is a semitoric
system for which $J$ is proper, then any cartographic projection of
$F$ is a convex polygon in $\mathbb{R}^2$, which may be bounded or
unbounded, and which is always a closed subset of $\mathbb{R}^2$; see
\cite[Theorem 3.8]{VN2007}.

  \begin{example}
    Figure~\ref{fig:cartographic2} shows the regular and singular
    focus\--focus fibers of singular Lagrangian fibration $f \circ F
    \colon M \to \bigcup_{j \in \Z} \mathcal{C}^{\vec \epsilon}_j$ in
    Theorem~\ref{theo:b}. There are two focus\--focus singular fibers,
    $F^{-1}(c_i)$, $i=1,2$.  The value $c_1$ has multiplicity $k_1=2$
    and $c_2$ has multiplicity $k_2=3$.
  \end{example}

  \begin{remark} \label{DH} Concerning
    Theorem~\ref{theo:b}(P.\ref{bp4}), note that the
    Duistermaat\--Heckman theorem does not hold for nonproper momentum
    maps. Indeed, let $M=S^2\times S^2$ with $F=(z_1,z_2)$ (toric
    momentum map). Let $f \colon [-1,1]\to (-1,1]$ be continuous. Let
    $M'=F^{-1}(\{(x,y) \,|\, x \in [-1,1],\,\, y<f(x)\}$. The set $M'$
    is an open subset of $M$, and $\mu=z_1$ is a momentum map for a
    Hamiltonian $S^1$\--action on $M'$. Furthermore, $\mu$ is not
    proper because $ \mu^{-1}(x)=F^{-1}(\{(x,y)\,\,|\,\, y<f(x)\}) $
    is not closed. Now let $ V(x)$ be the symplectic volume of $M'_x$
    where $M'_x=M'\cap \mu^{-1}(x)/S^1=S^2 \cap \{z_2<f(x)\}.$ (See
    Figure~\ref{fig:sph}.) Then $V(x)={\rm vol}(S^2) \,\,
    [(1+f(x))/2]=2\pi(1+f(x))$.  So $V(x)$ is not piecewise linear in
    general, in contrast with Duistermaat\--Heckman \cite{DH82}. This
    shows that the Duistermaat-Heckman theorem may not hold when the
    $S^1$-momentum map is not proper. Notice that the full map
    $F_{\restr M}$ is not proper, but we can easily modify it as
    follows.  Let $g(x,y)=(x,1/(f(x)-y))$. Then $F'=g \circ F|_{M'}$
    is proper, and the $S^1$-momentum map is not modified. Thus $F'$
    is a generalized semitoric system.
    \begin{figure}[h]
      \centering
      \includegraphics[width=0.4\textwidth]{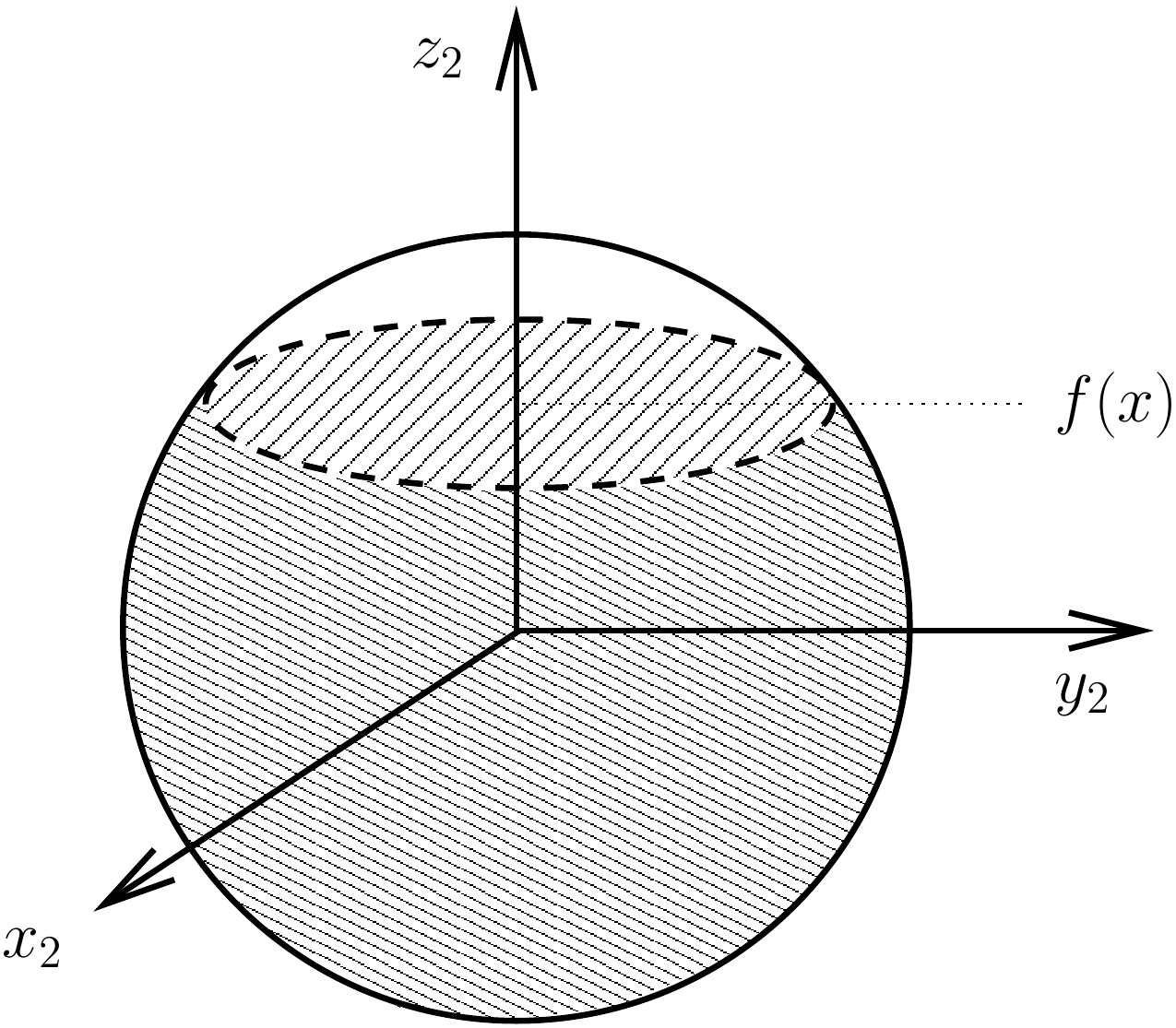}
      \caption{The reduced manifold $M'_x$.}
      \label{fig:sph}
    \end{figure}
  \end{remark}

  In the following definition, we use the terminology of
  sections~\ref{sec:toric-semitoric} and~\ref{sec:summary}.
  \begin{definition}
    Let $\mathcal{C}_{\rm GST}(F)$ be defined as follows. Let $\vec
    \epsilon = (\epsilon_i)_{i\in Z}$ with $\epsilon_i=1$ for all
    $i\in Z$. Then
    \begin{eqnarray} \label{deltaf} \mathcal{C}_{\rm GST}(F):=(f_{\vec
        \epsilon} (B) \,\,\, {\rm mod}\,\mathcal{T}) \mod \{-1,1\}^Z
      \in \mathfrak{B}_{\rm GST}.
    \end{eqnarray}
  \end{definition}
  \medskip

  \begin{theoremL} \label{theo:c}
    \label{prop:group}
    The map $\mathcal{C}_{\rm GST}:\mathcal{M}_{\rm GST}\to
    \mathfrak{B}_{\rm GST}$ is a cartographic invariant.
  \end{theoremL}

  \begin{proof}
    Let $F_1 \colon M_1 \to \mathbb{R}^2$ and $F_2\colon M_2 \to \R^2$
    be semitoric systems, and let $f_{\vec\epsilon,1}$,
    $f_{\vec\epsilon,2}$ be the corresponding cartographic maps
    defined by Theorem~\ref{theo:polygon}. If $F_1$ and $F_2$ are
    isomorphic, they have the same leaf space, with identical induced
    integral affine structures. Thus, from Theorem~\ref{theo:polygon},
    (P.\ref{pii}), there must be a transformation $t\in\mathcal{T}$
    such that $f_{\vec\epsilon,1}= t \cdot f_{\vec\epsilon,2}$. Then
    the result follows from~\eqref{deltaf}.
  \end{proof}

  We conclude with a result which shows that there are semitoric
  systems with a cartographic projection which may not occur as the
  cartographic projection of a toric or semitoric system $(J,H):M \to
  \mathbb{R}^2$ with proper $J$.

  \begin{theoremL} \label{niceexample} There exists an uncountable
    family of semitoric integrable systems $\{F_{\lambda}: M \to
    \mathbb{R}^2\}_{\lambda}$, with cartographic map $f_{\lambda,\vec
      \epsilon}$\,, such that the following properties hold:
    \begin{enumerate}[{\rm (E.1)}]
    \item \label{d1} $F_{\lambda} \colon M \to \mathbb{R}^2$ is
      proper;
    \item \label{d2}$B_{\lambda}:=F_{\lambda} (M)$ is unbounded in
      $\mathbb{R}^2$;
    \item \label{d3}$f_{\lambda,\vec \epsilon} (B_{\lambda})$ is a
      bounded in $\R^2$;
    \item \label{d4} $B_{\lambda}$ is not a convex region;
    \item \label{d5}$B_{\lambda}$ is not open and is not closed in
      $\R^2$;
    \item \label{d6} $F_{\lambda}$ is isomorphic to $F_{\lambda'}$ if
      and only of $\lambda=\lambda'$;
    \item \label{d7} for every $i \in \{I,II,III,IV\}$ there exists
      $\lambda$ such that $f_{\lambda,\vec \epsilon} (B)$ as in
      Theorem {\rm \ref{theo:b}{\rm (}P.\ref{pii}{\rm )}}, is a union
      of regions in $\R^2$ of types I, II, III, and IV, in which at
      least one of them has type $i$.
    \end{enumerate}
  \end{theoremL}
  
  Motivated by \cite{PeVN2010}, the following inverse type question is
  natural.  \emph{Let $C:=\cup_{j \in\mathbb{N}}C_j$ be a connected
    set, where $C_j \subset \mathbb{R}^2$ is a region of type I, II,
    III, or IV. Does there exist a semitoric system $F:M \rightarrow
    \mathbb{R}^2$ with $B:=F(M)$ such that $f_{\vec{\epsilon}}(B) =
    C$, where $f_{\vec{\epsilon}}$ is a cartographic map for $F$ ?  }

  The classifications of Delzant~\cite{Delzant1988} and
  \cite{PeVN2010} give partial answers to this question.  Note that
  here we are not claiming uniqueness; in fact, it follows
  from~\cite{PeVN2010} that there are many semitoric systems which
  realize the same $C$.

  \section{Proof of Theorem \ref{theo:polygon}} \label{sec:a}

  The proof is close to \cite{VN2007}, but our construction is more
  transparent thanks to the use of a recent result in
  \cite{PeRaVN2011}.  Let $\Sigma_J$ be the bifurcation set of $J$.
  We fix a point $q_0=(x_0,y_0) \in B_r$, such that $x_0 \notin
  \Sigma_J$. Since the fibers of $J$ are connected by
  (H.\ref{H-connected}), we know from \cite[Theorem~4.7]{PeRaVN2011}
  that the fibers of $F$ are also connected. By the
  Liouville-Mineur-Arnold Theorem (see, \cite[Appendix A2]{HoZe1994}),
  there exists a diffeomorphism $g:U\subset B_r \rightarrow g(U)
  \subset \mathbb{R}^2$, with positive Jacobian determinant, defined
  on an open neighborhood $U$ of $q_0$ which, without loss of
  generality, we may assume to be simply connected, such that $
  A=(A_1,A_2)=g \circ F $ are local action variables. Since $J$ is the
  momentum map of an effective Hamiltonian $S^1$-action, it has to be
  free on the regular fibers (see for instance~\cite[Theorem
  2.8.5]{DK00}). Hence, we may assume $A_1= J$ (see \cite[point 2 of
  the proof of Theorem 3.8]{VN2007}). Repeating this argument with an
  open cover of $B_r$, we may fix an affine atlas of $B_r$ such that
  all transition functions belong to the group $\mathcal{T}$ (see
  Section~\ref{sec:groupT}).

  We divide the proof into four steps: the first four treat the
  generic case in which the lines in $\ell^{\vec{\epsilon}}$ are
  pairwise distinct, whereas the last step deals with the
  non\--generic case.  We warn the reader that statements
  (P.\ref{pi})--(P.\ref{piii}) are proven in the first three steps,
  but the claim that $f_{\vec \epsilon}$ is a homeomorphism onto its
  open image is proven in Step 4.

  The homeomorphism $f_{\vec{\epsilon}}$ with the required properties
  is constructed from the developing map of the universal cover
  $p_r:\widetilde{B}_r\rightarrow B_r$, chosen with $q_0$ as base
  point (see Section~\ref{sec:affine}).  Let $\widetilde{G}_\epsilon:
  \widetilde{B}_r \rightarrow \mathbb{R}^2$ be the unique developing
  map such that $\widetilde{G}_\epsilon([\gamma])=g(\gamma(1))$ for
  paths $\gamma$ contained in $U$, such that $\gamma(0) = q_0$. The
  goal is to use $\widetilde{G}_\epsilon$ in order to extend $g$ to
  the whole image $B=F(M)$.

  We begin the proof by assuming that the half lines in $\ell^{\vec
    \epsilon}$ do not overlap.

  \paragraph{{\bf Step 1}} (\emph{$B_r\setminus \ell^{
      \vec{\epsilon}}$ is simply connected}).  By
  \cite[Theorem~4.7]{PeRaVN2011}, $B$ is a region of $\R^2$ which is
  between the graphs of two continuous functions defined on the same
  interval. These graphs cannot intersect above an interior point of
  the interval, because this would imply that the interior of $B_r$ is
  not connected, which is known to be false because $F$ is proper (see
  \cite[Theorem~3.6]{PeRaVN2011}). This proves that $B_r\setminus
  \ell^{\vec \epsilon}$ is simply connected.

  \paragraph{{\bf Step 2}} (\emph{Proof of {\rm (P.\ref{pi})} and {\rm
      (P.\ref{pii})} on $B_r \setminus \ell^{\vec \epsilon}$}).
  Hence, the developing map $\widetilde{G}_{\vec \epsilon}:
  \widetilde{B}_r \rightarrow \mathbb{R}^2$ induces a unique affine
  map $G_{\vec \epsilon}:B_r \setminus \ell^{\vec{\epsilon}}
  \rightarrow \mathbb{R}^2$ by the relation
$$G_{\vec{\epsilon}}\circ p_r:=\widetilde{G}_{\vec{\epsilon}},$$
i.e., if $c\in B_r \setminus \ell^{\vec{\epsilon}}$ and $\gamma$ is a
smooth path in $B_r \setminus \ell^{\vec{\epsilon}}$ connecting $q_0$
to $c$, then $G_{\vec{\epsilon}}(c):=
\widetilde{G}_{\vec{\epsilon}}([\gamma])$.  Note that
$G_{\vec{\epsilon}}|_U = g$.

The definition implies that $G_{\vec{\epsilon}}$ is a local
diffeomorphism.  We show now that $G_{\vec{\epsilon}}$ is
injective. Since $A_1=J$, $G_{\vec{\epsilon}}|_U$ is of the form
$G_{\vec{\epsilon}}(x,y)=(x,h^U_{\vec{\epsilon}}(x,y))$ for some
smooth function $h^U_{\vec \epsilon}: U \rightarrow
\mathbb{R}$. Because we have an affine atlas of $B_r$ with transition
functions in $\mathcal{T}$, the affine map $G_{\vec \epsilon}$ must
preserve the first component $x$, i.e.  there exists a smooth function
$h_{\vec{\epsilon}}: B_r \setminus \ell^{\vec{\epsilon}} \rightarrow
\mathbb{R}$, extending $h^U_{\vec{\epsilon}}$ such
that $$G_{\vec{\epsilon}}(x,y) = (x, h_{\vec{\epsilon}}(x,y))$$ for
all $(x,y) \in B_r \setminus \ell^{\vec{\epsilon}}$. Since
$G_{\vec{\epsilon}}$ is a local diffeomorphism, $\frac{\partial
  h_{\vec \epsilon}}{\partial y}$ never vanishes, which implies that
for each fixed $x$, all the maps $y \mapsto h_{\vec{\epsilon}}(x,y)$
are injective. Hence $G_{\vec{\epsilon}}$ is injective and thus a
global diffeomorphism $B_r\setminus \ell^{\vec{ \epsilon}} \rightarrow
G_{\vec{\epsilon}}\left(B_r\setminus \ell^{\vec{\epsilon}}\right)
\subset \mathbb{R}^2$.

This proves (P.\ref{pi}) on $B_r \setminus \ell^{\vec \epsilon}$ by
choosing $f_{\vec{\epsilon}}:= G_{\vec{\epsilon}}$ and (P.\ref{pii})
because $G_{\vec{\epsilon}}$ is an affine map.

\paragraph{{\bf Step 3}} (\emph{Extension of the developing map to $B
  \setminus \ell^{\vec \epsilon}$ and proof of {\rm (P.\ref{pi})} and
  {\rm (P.\ref{piii})}}).  By the description of the image of $F$ in
\cite[Theorem~5]{PeRaVN2011}, we simply need to extend $G_{\vec
  \epsilon}$ at elliptic critical values.  But the behavior of the
affine structure at an elliptic critical value $c$ is well known (see
\cite{MZ}): there exist a smooth map $a \colon V \to \mathbb{R}^2$,
where $V$ is an open neighborhood of $c \in \R^2$, and a
symplectomorphism $\varphi \colon F^{-1}(V) \to M_Q$ onto its image
such that
\begin{eqnarray} \label{star} a \circ F|_{F^{-1}(V)}=Q\circ \varphi: F
  ^{-1}(V) \rightarrow\mathbb{R}^2,
\end{eqnarray}
where $Q$ is the ``normal form" of the same singularity type as $F$,
given by $Q=(x_1^2+\xi_1^2,\xi_2)$ (rank $1$ case) or
$Q=(x_1^2+\xi_2^2,x_2^2+\xi_2^2)$ (rank $0$ case).  Here $M_Q =
\R^2\times{\rm T}^*\mathbb{T}^1 = \mathbb{R}^2 \times
\mathbb{T}^1\times \R$ (rank $1$) or $M_Q=\R^4$ (rank $0$).  It
follows from the formula for $Q$ that $Q$ is generated by a
Hamiltonian $\T^2$\--action, and therefore $a$ is an affine map. On
the other hand, since $F$ and $Q$ have the same singularity type, the
ranks of ${\rm d}F$ and ${\rm d}Q$ must be equal, and the dimensions of the spaces
spanned by the Hessians must be the same as well. Computing the Taylor
expansion of (\ref{star}) shows that ${\rm d}a(c)$ has to be invertible.
Thus, $a$ is a diffeomorphism onto its image.  Therefore $a|_{B_r \cap
  V}$ is a chart for the affine structure of $B_r$.

Thus there exists a unique affine map $A \in {\rm Aff}(2,\Z)$ such
that
$$
(G_{\vec \epsilon})|_{Br\cap V}=A \circ a|_{Br \cap V}
$$
and we may simply extend $G_{\vec \epsilon}$ to $B_r \cup V$ by
letting
$$
(G_{\vec \epsilon})|_{V}=A \circ a.
$$
Because $a$ is a diffeomorphism into its image, we see that
$G_{\vec{\epsilon}}$ remains a local diffeomorphism. This proves
(P.\ref{pi}) with $f_{\vec{\epsilon}}|_{B \setminus \ell^{\vec
    \epsilon}}:=G_{\vec \epsilon}$.

The fact that $G_{\vec \epsilon}$ extends to a smooth multi-valued map
$B_r \to \R^2$ follows from the smoothness of the universal cover as
in \cite[Section 3]{VN2007}. Formula (\ref{jump}) follows from the
calculation of the monodromy around focus-focus singularities, which
is carried out exactly as in \cite[pages 921-922]{VN2007} since it
relies only on the properness of $F$ (and not on the properness of
$J$).  This proves (P.\ref{piii}).

\paragraph{{\bf Step 4}} (\emph{Extension to a homeomorphism $B \to
  \R^2$}). Finally we show that $G_{\vec \epsilon}$ may be extended to
a homeomorphism $f_{\vec \epsilon} \colon B \to f_{\vec \epsilon}(B)
\subset \mathbb{R}^2$, which will prove the theorem if no half lines
in $\ell^{\vec{\epsilon}}$ overlap.

Because of (P.\ref{piii}), if $c_0 \in \ell^{\vec \epsilon}$, but
$c_0$ is not a focus-focus value, it follows that $G_{\vec \epsilon}$
has a unique continuation to $c_0$, from the left, and a unique
continuation from the right.  As in \cite[Proof of Theorem
3.8]{VN2007}, the fact that these continuations coincide follows from
the fact that the affine monodromy around a focus\--focus singularity
leaves the vertical line through $c_0$ pointwise invariant. That
$G_{\vec \epsilon}(c)$ has a limit as $c$ approaches the focus-focus
value follows from the $z\,{\rm log}z$ behavior of
$G_{\vec{\epsilon}}$, see \cite[Section 3]{VN2003}.

Let $f_{\vec \epsilon} \colon B \setminus \{c_i \mid i\in Z\} \to
\R^2$ be this continuous extension of $G_{\vec \epsilon}$. Because of
(P.\ref{piii}), the extensions of the vertical derivative
$\partial_yf_{\vec \epsilon}$ from the left or from the right coincide
on $\ell^{\vec \epsilon}$.  Since any extension of
$G_{\vec{\epsilon}}(x,y) = (x, h_{\vec{\epsilon}}(x,y))$ is a local
diffeomorphism, $\partial_yh_{\vec \epsilon}$ cannot vanish on
$\ell^{\vec \epsilon}$. Thus, $f_{\vec \epsilon}|_{\ell^{\vec
    \epsilon}}$ is injective.

This implies that $f_{\vec \epsilon}$ is injective on $B \setminus
\{c_i \mid i\in Z\}$.

Extend by continuity the map $f_{\vec \epsilon}$ to $\{c_i \mid i\in Z
\}$.  So far, we have shown that $f_{\vec{\epsilon}}:B \rightarrow
\mathbb{R}^2$ is a continuous injective map which is an affine
diffeomorphism off $\ell^{\vec{\epsilon}}$. It remains to be shown
that $(f_{\vec{\epsilon}})^{-1}$ is continuous on
$f_{\vec{\epsilon}}(B)$. Since $f_{\vec{\epsilon}}$ is a
diffeomorphism off $\ell^{\vec{\epsilon}}$, we only have to show that
$(f_{\vec{\epsilon}})^{-1}$ is continuous at points of
$f_{\vec{\epsilon}}(\ell^{\vec{ \epsilon}})$.

Let $c_0=(x_0,y_0) \in \mathring{\ell}^{\vec \epsilon}$ and
$\widehat{G}_{\vec \epsilon} \colon U \to \widehat{G}_{\vec
  \epsilon}(U)$ be an affine chart which coincides with $f_{\vec
  \epsilon}$ on the left hand\--side of $c_0$ in $U$, that is, on
$$U_{{\rm left}}:=
\Big\{(x,y) \in U \, \, | \,\, x \leq x_0\Big\}.
$$ 
Then,
$$
(f_{\vec \epsilon})^{-1}|_{f_{\vec \epsilon}(U_{{\rm left}})}
=\widehat{G}_{\vec \epsilon}^{-1}|_{f_{\vec \epsilon} (U_{{\rm
      left}})}
$$
and hence it is continuous on $f_{\vec \epsilon}( U_{{\rm
    left}})$. Similarly, it is proved that $(f_{\vec
  \epsilon})^{-1}|_{f_{\vec \epsilon}(U_{{\rm right}})}$ is continuous
on $U_{{\rm right}}$, which shows that $(f_{\vec{\epsilon}})^{-1}$ is
continuous at $f_{\vec{\epsilon}}(c_0)$ for any $c_0
\in\mathring{\ell}^{\vec \epsilon}$.

Finally, we need to prove the continuity of
$(f_{\vec{\epsilon}})^{-1}$ at all points $f_{\vec{\epsilon}}(c_i)$,
where $c_i=(x_i,y_i)$, $i\in Z$, are the focus-focus values in $B$.
Let $\ell_i$ be the vertical line containing $c_i$.  Let us use the
following local description of the behavior of $f_{\vec \epsilon}$ at
$c_i$, \cite{VN2003}, \cite[Proof of Theorem~3.8]{VN2007}: for all
$(x,y) \in U\setminus \ell_i$,
$$
f_{\vec \epsilon}(x,y)=(x,\, {\rm Re}(z{\,\rm log}z)+g(x,y)),
$$
where $z=\hat{y}(x,y)+ {\rm i} x \in \mathbb{C}$, $g$ and $\hat{y}$
are smooth functions and $\hat{y}(0,0)=y_0$. It follows that
$\frac{\partial f_{\vec{\epsilon}}}{\partial y}$ is continuous near
$c_i$ (which is in agreement with~\eqref{jump}) and is equivalent, as
$z\to 0$, to $K\ln(x^2+y^2)$ for some constant $K>0$. Hence we get the
lower bound $$\left|\frac{\partial f_{\vec{\epsilon}}}{\partial
    y}\right| \geq C>0$$ for some constant $C$, if $(x,y)$ is in a
small neighborhood $V=[x_i-\eta,x_i+\eta]\times[y_i-\eta,y_i+\eta]$ of
$c_i$, for some $\eta>0$. For simplicity of notation, let us assume
for instance that $\epsilon_i=1$; the case $\epsilon_i=-1$ is treated
similarly.  Hence, for any fixed $x\in[x_i-\eta,x_i+\eta]$, the
function $y \mapsto f_{\vec{\epsilon}}(x, y)$ is invertible on
$(y_i,y_i+\eta]$ and has bounded derivative, uniformly for
$x\in[x_i-\eta, x_i+\eta]$. Hence, the inverse
$(f_{\vec{\epsilon}})^{-1}$ extends by continuity at $f(c_i) =
f(x_i,y_i)$. The limit of the inverse at this point must equal $y_i$
since $f_{\vec{\epsilon}}$ is injective.  This shows that
$(f_{\vec{\epsilon}})^{-1}$ is continuous at the point
$f_{\vec{\epsilon}}(c_i)$.

This concludes the proof of Theorem~\ref{theo:polygon} in case there
is no overlap of vertical lines in $\ell^{\vec \epsilon}$.

\paragraph{{\bf Step 5}} (\emph{Proof in the case of overlapping lines
  in $\ell^{\vec \epsilon}$}).  If on the other hand there are
overlaps of vertical lines in $\ell^{\vec \epsilon}$, then $B_r
\setminus \ell^{\vec \epsilon}$ may not be simply connected.  In this
case, for each $c\in B_r \setminus \ell^{\vec \epsilon}$, we need to
choose a path $\gamma_c$ joining $q_0$ to $c$ inside
$B_r\setminus\{c_i\mid i\in Z\}$, which we do as follows.  We replace
the focus-focus critical values $c_i$ which lie in the same vertical
line by nearby points $\widetilde{c}_i$, in such a way that their
$x$-coordinates are all pairwise distinct.  This turns the
corresponding set $B_r \setminus \tilde{\ell}^{\vec{\epsilon}}$ into a
simply connected set; thus, up to homotopy, there is a unique path
$\gamma_c$ joining $q_0$ to $c$ inside $B_r
\setminus\tilde{\ell}^{\vec \epsilon}$, and we can always assume that
this path avoids the true focus-focus values $c_i$.

The homotopy class of $\gamma_c$ depends on the choice of ordering of
the $x$-coordinates of the points $\widetilde{c}_i$. However, we claim
that the value $$G_{\vec{\epsilon}}(c):=
\widetilde{G}_{\vec{\epsilon}}([\gamma_c])$$ is well defined. Indeed,
decomposing a permutation as a product of transpositions of the form
$(i,i+1)$ or $(i+1,i)$, it suffices to consider only the case where we
permute two points, $\widetilde{c}_i$ and $\widetilde{c}_{i+1}$, which
lie in adjacent vertical lines. In this case, one can check that the
homotopy class $[\gamma_c]$ is modified by a commutator
$g_ig_{i+1}g_i^{-1}g_{i+1}^{-1}$, where $g_i$, $i\in Z$, is a set of
generators of the fundamental group of $B_r\setminus\{c_i \mid i \in
Z\}$. But the monodromy representation is Abelian, due to the global
$S^1$ action (see~\cite{cushman-vungoc}).  It follows that, as
required, the value $\widetilde{G}_{\vec \epsilon}([\gamma_c])$ is
invariant under this transposition.

Now that $G_{\vec \epsilon}$ is defined, the previous proof for
(P.\ref{pi}) and (P.\ref{pii}) remains valid. The formula in
(P.\ref{piii}) follows from the fact that the monodromy representation
is Abelian.

\section{Proof of Theorem \ref{theo:b} and the spherical pendulum
  example} \label{sec:b}

\begin{proof}[Proof of Theorem \ref{theo:b}]

  The proof is divided into three steps.

  \paragraph{{\bf Step 1}} Let $f_{\vec \epsilon} \colon B \to f_{\vec
    \epsilon} (B) \subset \mathbb{R}^2$ be the homeomorphism in
  Theorem~\ref{theo:polygon}.  Let $H^{+},\, H^{-} \colon J(M) \to
  \overline{\RM}$ be the functions defined by $H^{+}(x) :=
  \sup_{J^{-1}(x)} H$ and $H^{-}(x) :=\inf_{J^{-1}(x)} H$.  Since $J$
  is Morse-Bott with connected fibers (see, e.g.,
  \cite[Theorem~3]{PeRa11}) we may apply \cite[Theorem
  5.2]{PeRaVN2011} which states that $H^{+},H^{-}$ are continuous and
  $ F(M)=\left(\text{hypograph~of~}H^+\right) \cap
  \left(\text{epigraph~of~}H^-\right).  $ Since $H^{+},H^{-}$ are
  continuous and $F$ is proper, one can check that the sets
  $K^+,K^{-}$ defined in the theorem are open in $J(M)$. Hence we have
  the following equality of sets, where the four sets on the right
  hand side are open and disjoint: $ J(M)=(K^+\cap K^{-}) \cup
  (K^+\setminus K^{-}) \cup (K^{-}\setminus K^{+}) \cup ( J(M)
  \setminus (K^+\cup K^{-})).$ By assumption, $\partial K^{+}$ and
  $\partial K^{-}$ are discrete, and therefore there exists a
  countable collection of intervals $\{I_j\}_{j\in \Z}$, whose
  interiors are pairwise disjoint, such that each $I_j$ is contained
  in one of the above four sets $(K^+\cap K^{-})$, $(K^+\setminus
  K^{-})$, $(K^{-}\setminus K^{+})$ or $( J(M) \setminus (K^+\cup
  K^{-}))$, and such that $ J(M)=\bigcup_{j \in \Z}I_j $.

  By letting for every $j \in \Z$, $\mathcal{C}^{\vec
    \epsilon}_j:=f_{\vec \epsilon} ((I_j \times \R) \cap F(M)) \subset
  I_j \times \mathbb{R}$, we obtain $ f_{\vec
    \epsilon}(F(M))=\bigcup_{j \in \Z} \mathcal{C}^{\vec \epsilon}_j,
  $.

  \medskip

  \paragraph{{\bf Step 2}} (\emph{Proof of {\rm (P.\ref{bp1})} and
    {\rm (P.\ref{bp2})}}). We consider the four cases.
  \begin{enumerate}
  \item \label{n1} If $I_j \subset (K^+\cap K^{-})$, then the fibers
    of $J$ are compact, and hence the analysis carried out in
    \cite[Theorem~3.8, (v)]{VN2007} applies. This implies that
    $\mathcal{C}^{\vec \epsilon}_j$ is of type I.

  \item\label{n2} Consider now $I_j \subset (K^-\setminus K^{+})$.
    Let $x \in I_j$. Since $J^{-1}(x)\cap H^{-1}((-\infty,0])$ is
    compact, $H_-(x)$ is finite. On the other hand, $H_+(x)$ must be
    $+\infty$; otherwise, $F^{-1}(\{x\}\times [0,H_+(x)])$ would be
    compact, by the properness of $F$. This would imply that
    $J^{-1}(x)$ is compact, a contradiction.

    Let $y \in H(J^{-1}(x))$.  Recall that $
    f_{\vec{\epsilon}}(x,y)=(x, f_{\vec{\epsilon}}^{(2)}(x,y)) $ and
    that $\frac{\partial f_{\vec{\epsilon}}^{(2)}}{\partial y}$ is
    continuous on $F(M)$ (see~\eqref{jump}). Since $\frac{\partial
      f_{\vec{\epsilon}}^{(2)}}{\partial y}>0$, the image $ f_{\vec
      \epsilon} ((I_j\times \R) \cap F(M))= \mathcal{C}^{\vec
      \epsilon}_j $ has the form
$$
\Big \{(x,z) \mid x \in I_j,\,\,\, h^{\vec \epsilon}_{-}(x)\leq z <
h^{\vec{\epsilon}}_+(x) \Big\},
$$
where
\begin{align*}
  h^{\vec{\epsilon}}_{-}(x)&:=\min_{y \in J^{-1}(x)}
  f_{\vec{\epsilon}}^{(2)}(x,y)
  =f_{\vec{\epsilon}}^{(2)}(x,H_{-}(x)) \in\mathbb{R}\\
  h^{\vec{\epsilon}}_{+}(x)&:=\sup_{y \in J^{-1}(x)}
  f_{\vec{\epsilon}}^{(2)}(x,y)= \lim_{y \to +\infty}
  f^{(2)}_{\vec{\epsilon}}(x,y) \in \overline{\R}.
\end{align*}
We have used the fact that $f_{\vec{\epsilon}}$ is a homeomorphism, so
that the point $(x,h^{\vec{\epsilon}}_{+}(x))$ cannot belong to
$\mathcal{C}^{\vec \epsilon}_j$. The function $h_{+}^{\vec\epsilon{}}$
is a pointwise limit of continuous functions, so it is continuous on a
dense set. However, we need to show that it is lower semicontinuous.

The new map
$$
\left(J,f_{\vec{\epsilon}}^{(2)}(J,H)\right)=f_{\vec{\epsilon}} \circ
F
$$ 
satisfies the hypothesis of the following slight variation of
\cite[Theorem~5.2]{PeRaVN2011} for continuous maps (the proof of which
is identical line by line): \emph{Let $\widehat{M}$ be a connected
  smooth four-manifold. Let $\widehat
  F=(\widehat{J},\,\widehat{H}):\widehat M \to \mathbb{R}^2$ be a
  continuous map.  Suppose that the component $\widehat{J}$ is a
  smooth non-constant Morse-Bott function with connected fibers.  Let
  $\widehat{H}^{+},\, \widehat{H}^{-} :
  \widehat{J}\left(\widehat{M}\right) \to \overline{\mathbb{R}}$ be
  defined by $\widehat{H}^{+}(x) := \sup_{\widehat
    J^{-1}(x)}\widehat{H}$ and $\widehat{H}^{-}(x)
  :=\inf_{\widehat{J}^{-1}(x)} \widehat{H}$. Then the functions
  $\widehat{H}^{+}$ and $-\widehat{H}^{-}$ are lower semicontinuous.
} This statement gives the required semicontinuity in the statement of
Theorem~\ref{theo:b}.

The analysis of the graph of $h^{\vec \epsilon}_{-}$, which
corresponds to the elliptic critical values and possible cuts due to
focus-focus singularities, was carried out in
\cite[Theorem~3.8]{VN2007}: it is continuous, piecewise linear, and
convex. Thus, $\mathcal{C}^{\vec \epsilon}_j$ is of type II.

\item\label{n3} The fact that $I_j \subset (K^{+}\setminus K^{-})$
  implies that $\mathcal{C}^{\vec \epsilon}_j$ is of type III can be
  proved in a similar way to (\ref{n2}).
\item\label{n4} Finally, let $I_j \subset J(M) \setminus (K^+\cup
  K^{-})$. In this case, we must have, for any $x\in I_j$,
  $H_+(x)=+\infty$ and $H_-(x)=-\infty$.  Therefore, $ f_{\vec
    \epsilon} ((I_j\times \R) \cap F(M))= \mathcal{C}^{\vec
    \epsilon}_j $ has the form
$$
\Big \{(x,z) \mid x \in I_j,\,\,\, \lim_{y \to -\infty}
f_{\vec{\epsilon}}^{(2)}(x,y) < z < \lim_{y\to+ \infty}
f_{\vec{\epsilon}}^{(2)}(x,y)\Big\},
$$
where the limits are understood in $\overline{\R}$. Thus,
$\mathcal{C}^{\vec \epsilon}_j$ is of type IV.
\end{enumerate}

This proves (P.\ref{bp1}).

\paragraph{{\bf Step 3}} (\emph{Proof of {\rm (P.\ref{bp4}))}.} By the
action\--angle theorem, $(A_1,A_2):=f_{\vec \epsilon} \circ F$ is a
set of action variables near $F^{-1}(x,y)$ with
$$
A_1=J,\,\,\,\,\,\,\,\,\, A_2=A_2(J,H).
$$
We have a symplectomorphism $\mathcal{U} \to \T^2_{\theta}\times
\R^2_A $, where $\mathcal{U}$ is a saturated neighborhood of the fiber
$F^{-1}(x,y)$, and the symplectic form on $\T^2_{\theta}\times \R^2_A
$ is given by ${\rm d}A_1 \wedge {\rm d}\theta_1 +{\rm d}A_2 \wedge
{\rm d}\theta_2$. We have
$$
\mathcal{U} \cap J^{-1}(x)=A^{-1}_1(x)=\Big\{(\theta,A)\mid \theta \in
\T^2,\,\, A_1=x\Big\}.
$$
Since the normalized Liouville volume form is $(2\pi)^{-2}{\rm
  d}A_1\wedge {\rm d}A_2 \wedge {\rm d}\theta_1 \wedge {\rm
  d}\theta_2$, the induced volume form on $\mathcal{U}\cap J^{-1}(x)$
is $(2\pi)^{-2}{\rm d}A_2 \wedge {\rm d}\theta_1 \wedge {\rm
  d}\theta_2$.  In other words, the push\--forward by $A_2$ of the
Liouville measure on $J^{-1}(x)$ has a constant density $1$ against
the Lebesgue measure ${\rm d}A_2$.  This gives the result because the
set of critical points of $H$ in $J^{-1}(x)$ has zero\--measure in
$J^{-1}(x)$. This concludes the proof of Theorem~\ref{theo:b}.
\end{proof}

\begin{example}(Spherical Pendulum) \label{sph} Semitoric systems with
  proper $F=(J, H)$ but non-proper $J$ include many simple integrable
  systems from classical mechanics, such as the spherical pendulum,
  which we now recall. The phase space of the spherical pendulum is $M
  = {\rm T}^\ast S^2$ with its natural exact symplectic form. Let the
  circle $S^1$ act on the sphere $S^2\subset\R^3$ by rotations
  about   the vertical axis.  Identify ${\rm T}^\ast S^2$ with ${\rm T}S^2$,
  using the standard Riemannian metric on $S^2$, and denote its points
  by $({q},{p})= (q^1,q^2,q^3,p_1,p_2,p_3) \in {\rm T}^\ast S^2 = {\rm
    T} S^2$, $\|{q}\|^2=1$, $q\cdot p = 0$. Working in units in which
  the mass of the pendulum and the gravitational acceleration are
  equal to one, the integrable system $F:=(J,H) \colon {\rm T}S^2 \to
  \R^2$ is given by the momentum map of the (co)tangent lifted
  $S^1$-action on ${\rm T}S^2$,
  \begin{eqnarray} \label{J} J(q^1,q^2,q^3,p_1,p_2,p_3) = q^1p_2 -
    q^2p_1,
  \end{eqnarray}
  and the classical Hamiltonian
  \begin{eqnarray} \label{H} H(q^1,q^2,q^3, p_1,p_2,p_3)=
    \frac{(p_1)^2+(p_2)^2+(p_3)^2}{2} +q^3,
  \end{eqnarray}
  the sum of the kinetic and potential energy. 
  The momentum map $J$ is not proper because the sequence $\{(0,0,1,
  n,n,0)\}_{n \in \mathbb{N}}\subset J^{-1}(0) \subset {\rm T} S^2$
  does not contain any convergent subsequence.  The Hamiltonian $H$ is
  proper since $H^{-1}([a, b])$ is a closed subset of the compact
  subset of ${\rm T}S^2$ for which $2(a-1) \leq \| p\|^2 \leq
  2(b+1)$. Therefore, $F$ is also proper. In this case, $F(M)$ is
  depicted in Figure~\ref{critical_set_spherical_pendulum.figure} and
  the cartographic invariant of $(M,F)$ is represented in Figure
  \ref{fig:pendulum-polygon}; we call it $\Delta(F)$.

\begin{figure}[h]
  \centering
  \includegraphics[width=0.3\textwidth]{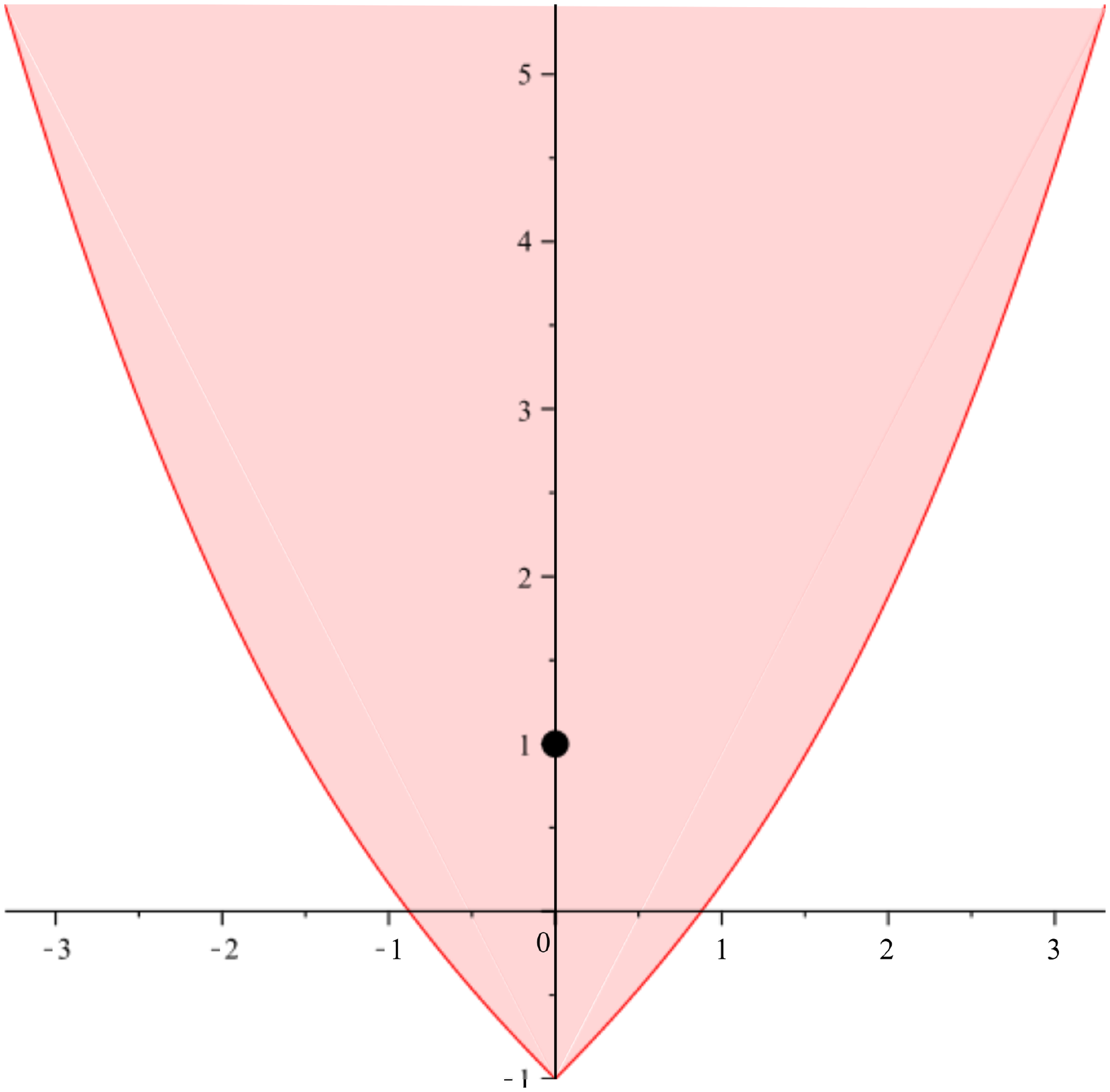}
  \caption{Image of of $F:=(J,H)$ given by (\ref{J}) and
    (\ref{H}). The edges are the image of the transversally-elliptic
    singularities (rank 1), the vertex is the image of the
    elliptic-elliptic singularity (rank 1), and the dark dot in the
    interior is the image of the focus-focus singularity (rank 0).
    All other points are regular (rank 2).}
  \label{critical_set_spherical_pendulum.figure}
\end{figure}

\begin{figure}[h]
  \centering
  \includegraphics[width=0.35\textwidth]{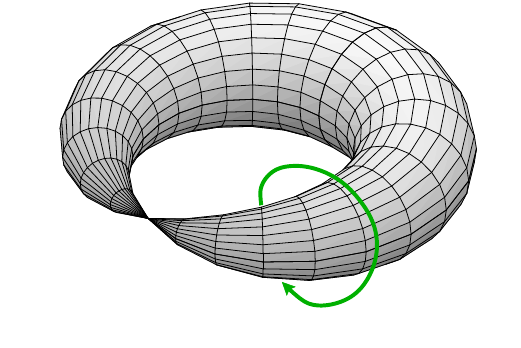}
  \caption{Fiber of $F:=(J,H)$ given by (\ref{J}) and (\ref{H}) over
    the focus\--focus critical value $(0,1)$.}
  \label{focusfocus.figure}
\end{figure}

There is precisely one elliptic-elliptic singularity at
$((0,0,-1),(0,0,0))$, one focus-focus singularity at
$((0,0,1),(0,0,0))$, and uncountably many transversally-elliptic type
singularities.  The range $F(M)$ and the set of critical values of
$F$, which equals its bifurcation set, are given in Figure
\ref{critical_set_spherical_pendulum.figure}.  The image under $F$ of
the focus-singularity is the point $(0,\,1)$.  The image under $F$ of
the elliptic-elliptic singularity is the point $(0,\,-1)$. We know
that the image by $J$ of critical points of $F$ of rank zero is the
singleton $\{0\}$.  Hence (one of the two representatives of)
$\Delta(F)$ has no vertex in both regions $J<0$ and $J>0$.  In each of
these regions, there is only one connected family of transversally
elliptic singular values. This means that $\Delta(F)$ in these region
consist of a single (semi-infinite) edge. We can arbitrarily assume
that, in the region $J<0$, the edge in question is the negative real
axis $\{(y,0)\mid y<0\}$. Then we have a vertex at the origin
$(x=0,y=0)$.
\begin{figure}[h]
  \centering
  \includegraphics[width=0.4\textwidth]{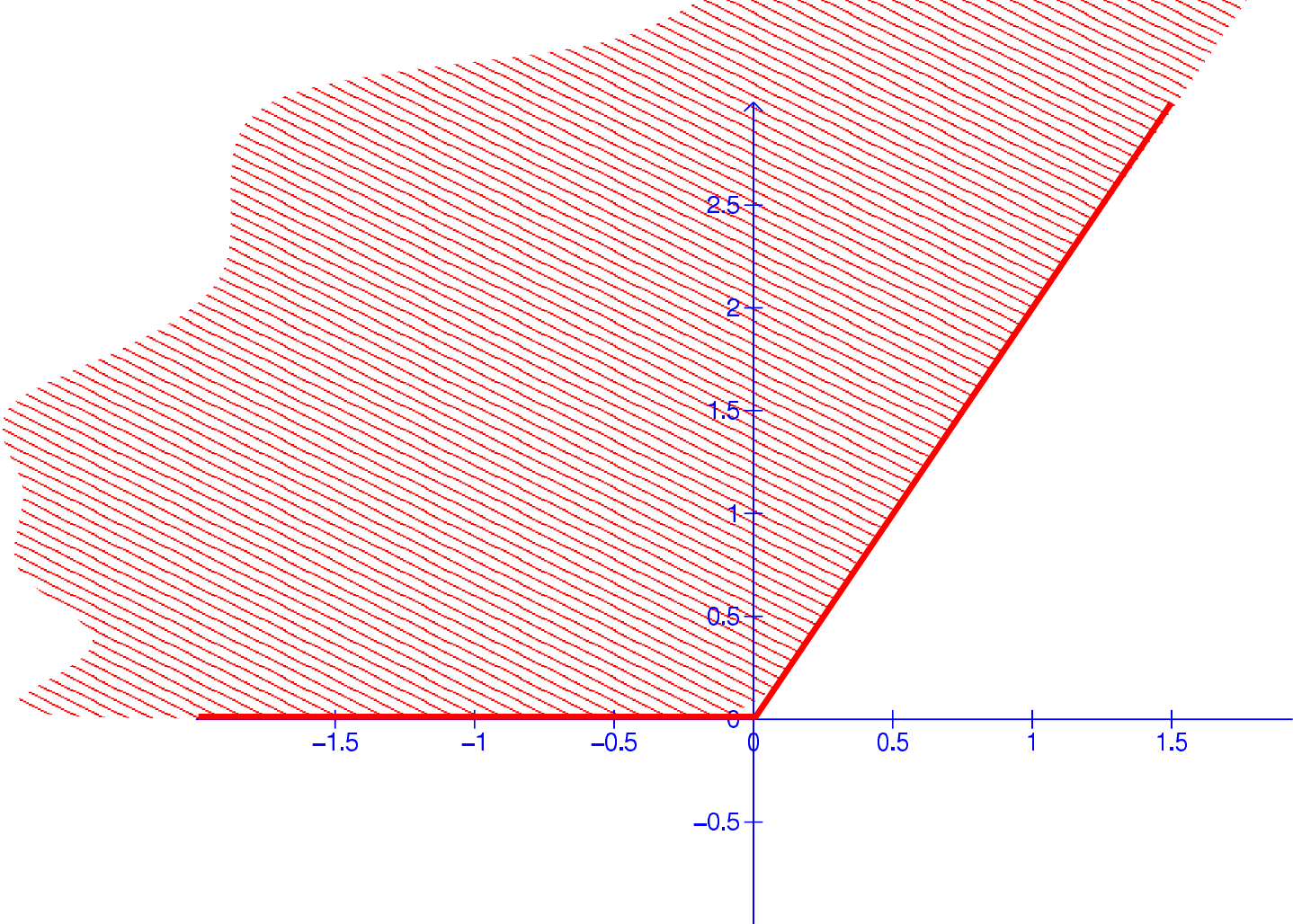}
  \caption{One of the two cartographic projections of the spherical
    pendulum.}
  \label{fig:pendulum-polygon}
\end{figure}
\end{example}

We still need to compute the slope of the edge corresponding to the
region where $J>0$. For this, we apply~\cite[Theorem 5.3]{VN2007},
which states that the change of slope can be deduced from the isotropy
weights of the $S^1$ momentum map $J$ and the monodromy index of the
focus-focus point.  (We need to include the focus-focus point because
its $J$-value is the same as the $J$-value of the elliptic-elliptic
point.) So we compute these weights now. The vertex of the polygon
corresponds to the stable equilibrium at the South Pole of the
sphere. We use the variables $(q_1,q_2,p_1,p_2)$ as canonical
coordinates on the tangent plane to the South Pole. In these
coordinates, the quadratic approximation of $J$ is in fact exact, and
equal to $J^{(2)}=q^1p_2-q^2p_1$. Now consider the following change of
coordinates:
\begin{equation}
  \begin{aligned}
    x_1 := \frac{q_2-q_1}{\sqrt{2}},\;\;\; x_2 :=
    \frac{p_1+p_2}{\sqrt{2}},\;\;\; \xi_1 :=
    \frac{p_1-p_2}{\sqrt{2}},\;\;\; \xi_2 := \frac{q_1+q_2}{\sqrt{2}}.
  \end{aligned}
  \label{equ:change}
\end{equation}
This is a canonical transformation and the expression of $J^{(2)}$ in
these variables is $J^{(2)} = \frac{1}{2}(x_2^2+\xi_2^2) -
\frac{1}{2}(x_1^2+\xi_1^2)$.  Since the Hamiltonian flows of
$\frac{1}{2}(x_2^2+\xi_2^2)$ and $\frac{1}{2}(x_1^2+\xi_1^2)$ are
$2\pi$-periodic, this formula implies that the isotropy weights of $J$
at this critical point are $-1$ and $1$. From~\cite{VN2007}, we know
that the difference between the slope of the edge in $J>0$ and the
slope of the edge in $J<0$ must be equal to $\frac{-1}{ab} + k$, where
$a$ and $b$ are the isotropy weights, and $k$ is the monodromy index.
For the spherical pendulum, $k=1$ because there is only one simple
focus-focus point. Thus the new slope is $\frac{-1}{ab} +
k=1+1=2$. This leads to the polygonal set depicted in
Figure~\ref{fig:pendulum-polygon}.

\section{Proof of Theorem \ref{niceexample}}

We give here the outline of the construction of a family of integrable
systems defined on an open subset of $S^2\times S^2$, leading to the
proof of Theorem~\ref{niceexample}.

\paragraph{{\bf Step 1}} (\emph{Construction of suitable smooth
  functions.}) Let $$\Omega:=[-1,\,1] \times [-1,\,1] \setminus \{0\}
\times [0,\,1].$$ Let $\chi \colon [-1,\,1] \to \mathbb{R}$ be any
$\op{C}^{\infty}$-smooth function such that $\chi(z_2)\equiv 1$ if
$z_2 \le 0$ and $0 < \chi(z_2) \neq 1$ if $z_2>0$. Define $f \colon
\Omega \to \mathbb{R}$ by
\begin{eqnarray}
  \label{function_f}  
  f(z_1,\,z_2)=
  \left\{ \begin{array}{rl}
      1 & \textup{ if } z_1\le 0; \\
      \chi(z_2) & \textup{ if } z_1>0. 
    \end{array} 
  \right.
\end{eqnarray}
and note that it is smooth on $\Omega$.

\paragraph{{\bf Step 2}.} (\emph{Definition of a connected smooth
  $4$-manifold $M$.})  Let $S^2$ be the unit sphere in $\mathbb{R}^3$
and $M:=S^2 \times S^2 \setminus \{((x_1,y_1,z_1),\,(x_2,y_2,z_2)) \in
S^2 \times S^2 \mid z_1=0,\, z_2 \ge 0\}$, where a point in the first
sphere has coordinates $(x_1,\,y_1,\,z_1)$ and a point in the second
sphere has coordinates $(x_1,\,y_2,\,z_2)$.  Since $M \subset S^2
\times S^2$ is an open subset, it is a smooth manifold. Moreover, $M$
is connected.

\paragraph{{\bf Step 3.}} (\emph{Definition of a smooth $2$-form
  $\omega \in \Omega^2(M)$.})  Let $\pi_i \colon S^2 \times S^2 \to
S^2$ be the projection on the $i^{th}$ copy of $S^2$, $i=1,\,2$.  Let
$\omega_i:=\pi^*\omega_{S^2}$ where $\omega_{S^2}$ is the standard
area form on $S^2$. Define the $2$\--form $\omega$ on $M$ by
\begin{equation}
  \label{ex_omega}
  \omega_{(m_1,\,m_2)}:=(\omega_1)_{m_2} +
  f(z_1,\,z_2) \, (\omega_2)_{m_2}
\end{equation}
for every $(m_1,\, m_2) \in M$. Since $f$ is smooth by Step 1,
$\omega$ is also smooth, i.e., $\omega \in \Omega^2(M)$.

\paragraph{{\bf Step 4}} (\emph{The $2$-form $\omega$ is symplectic.})
One can check that $\omega$ is closed because $\deriv{f}{z_1}=0$, and
that $\omega$ is non-degenerate because $f\neq 0$.

\paragraph{{\bf Step 5}} (\emph{$(M,\, \omega)$ with $J:=z_1,\,
  H:=z_2$ satisfies $\{J,\,H\}=0$ and $J$ is a momentum map for a
  Hamiltonian $S^1$-action.}) We let $S^1$ act on $M$ by rotation
about the (vertical) $z_1$-axis of the first sphere and trivially on
the second sphere. The infinitesimal generator of this action equals
the vector field $\mathcal{X}((x_1,y_1,z_1),(x_2,y_2,z_2))=
((-y_1,x_1,0), (0,0,0))$. This immediately shows that $J=z_1$ is a
momentum map for this action.

\paragraph{{\bf Step 6}} (\emph{$(M,\, \omega)$ with $J:=z_1,\,
  H:=z_2$ is a generalized semitoric system with only elliptic
  singularities.}) A direct verification shows that the rank zero
critical points are precisely $(N_1,N_2)$, $(N_1,S_2)$, $(S_1,N_2)$,
and $(S_1,S_2)$, where $N_i$, $S_i$ are the North and South Poles on
the first and second spheres, respectively. One can verify that these
critical points are non-degenerate, in the sense that a generic
combination $({\rm i}a,-{\rm i}a,{\rm i}bf(1),-{\rm i}bf(1))$ of the
linearizations of the vector fields $\mathcal{X}_J$ and
$\mathcal{X}_H$ at each of these points $(0,0,{\rm i}f(1),-{\rm
  i}f(1))$ and $(0,0,{\rm i},-{\rm i})$ has four distinct
eigenvalues. Thus these singularities are of elliptic-elliptic type.
The rank one critical points are $(N_1,(x_2,y_2,z_2))$,
$(S_1,(x_2,y_2,z_2))$, $((x_1,y_1,z_1),N_2)$ with $z_1\neq0$, and
$((x_1,y_1,z_1),S_2)$.  Another simple computation shows that all of
them are non-degenerate and of transversally elliptic type. It follows
that $J:=z_1$, $H:=z_2$ is an integrable system with only
non-degenerate singularities, of either elliptic-elliptic or
transversally elliptic type. Hence $(J:=z_1,\,H:=z_2)$ is a
generalized semitoric system.

Since the range of $F$ is
\begin{equation}
  \label{image_of_F}
  F(M)=[-1,\,1]\times [-1,\,1] \setminus 
  \{ z_1 =0, \, z_2\ge 0\},
\end{equation} 
is not a closed set (see also Figure \ref{fig:Fig_im_F}), it follows
that $F$ is not a proper map.

\begin{figure}[h]
  \centering
  \includegraphics[height=4cm]{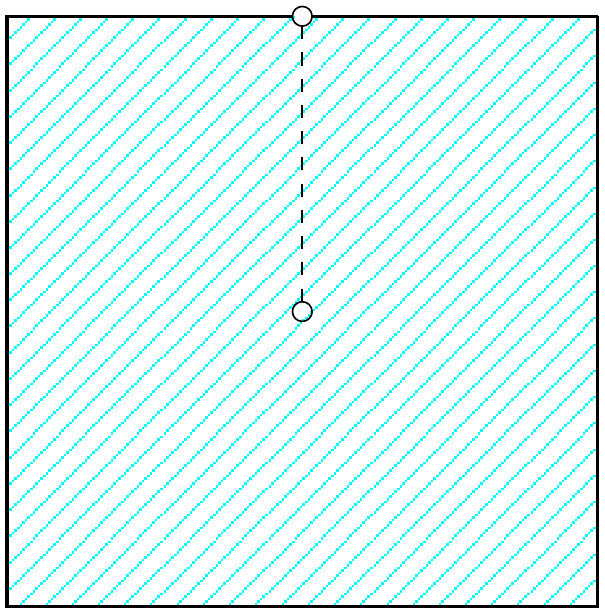}
  \caption{The image $F(M)$.}
  \label{fig:Fig_im_F}
\end{figure}

\paragraph{{\bf Step 7}} (\emph{Modify $F$ suitably to turn it into a
  proper map which still defines a semitoric system.})  Consider the
smooth function $g:\Omega \to \mathbb{R}^2$ defined by $g(z_1,\, z_2)=
\left(z_1,\, \frac{z_2+2}{z_1^2+h(z_2)}\right)$, where $h(z_2)\geq 0$,
$h(z_2)=0$ if and only if $z_2\geq 0$, and $h'(z_2) < 0$ for $z_2<0$.
Define $\widetilde{F}:=F \circ g= \left(J,\, \frac{H+2}{J^2+h(H)}
\right) \colon M \to \mathbb{R}^2$.  Since the Jacobian of
$\widetilde{F}$ is $$\frac{1}{(z_1 ^2+ h(z_2))^2} \left(z_1^2 + h(z_2)
  - h'(z_2)(z_2+2)\right)>0$$ (recall that $h'(z_2) \leq 0$ and $z_1^2
+ h(z_2)>0$ for $(z_1, z_2 \in\Omega$), it follows that
$\widetilde{F}$ is a local diffeomorphism.  In order to show that
$\widetilde{F}$ is proper, it suffices to prove that
$\widetilde{F}^{-1}(K_1 \times K_2)$ is compact if $K_1$ and $K_2$ are
closed intervals of $\mathbb{R}$; since the second component of $g$ is
always positive, we can assume, without loss of generality, that
$K_2=[a,b]$ with $a>0$. To show that $\widetilde{F}$ is proper, we
begin by analyzing $g^{-1}(K_1 \times K_2)$. We have $(z_1, z_2) \in
g^{-1}(K_1 \times K_2)$ if and only if $z_1 \in K_1$ and $0< a \leq
\frac{z_2+2}{z_1^2 + h(z_2)}\leq b$, which is implies
that $$\frac{1}{b} \leq \frac{z_2+2}{b} \leq z_1^2 + h(z_2).$$ Hence
either $z_1^2 \geq 1/2b$ or $h(z_2) \geq 1/2b$. Thus the set
$g^{-1}(K_1 \times K_2)$ lies inside the set $\Omega_b$ in Figure
\ref{FigD}.
\begin{figure}[h]
  \centering
  \includegraphics[height=4.5cm]{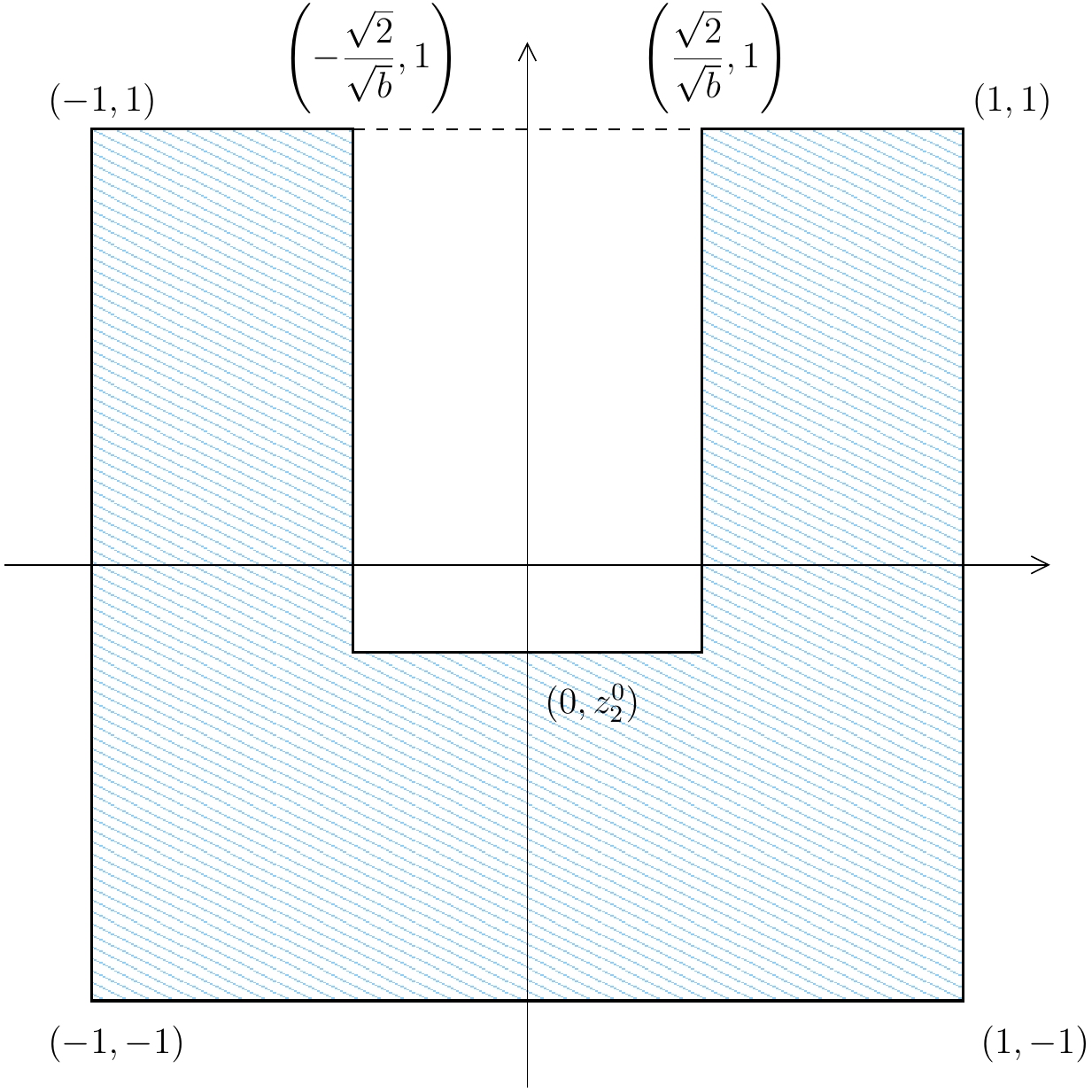}
  \caption{The set $\Omega_b$, where $z_2^0 <0$ is uniquely determined
    by the condition $h(z_2^0) = 1/2b$.}
  \label{FigD}
\end{figure}
Since $g^{-1}(K_1 \times K_2)$ is closed and obviously bounded, as a
subset of the compact set $\Omega_b$, it follows that
$g^{-1}(K_1\times K_2)$ is compact in $\mathbb{R}^2$. Therefore,
$$\widetilde{F}^{-1}(K_1 \times K_2) =
F^{-1}\left(g^{-1}(K_1 \times K_2)\right)$$ is compact in $S^2 \times
S^2$ and is obviously contained in $M$, by construction.  We conclude
that $\widetilde{F}^{-1}(K_1\times K_2)$ is compact in $M$, endowed
with the subspace topology.

Note that $J$ is not proper because $J^{-1}(0)$ is not
compact. However, $\widetilde{F}$ is a general semitoric system and
$\widetilde{F}$ is proper.

\paragraph{{\bf Step 8}.} (\emph{Finding the image
  $\widetilde{F}(M)$.}) Let
$$X:= \Big([-1, 0) \times [-1,1]\Big) \cup 
\Big((0,1]\times [-1,1]\Big) \cup \Big(\{0\} \times [-1,0)\Big).$$ It
follows from \eqref{image_of_F} (see also Figure \ref{fig:Fig_im_F})
that
\begin{align*}
  \widetilde{F}(M) = g(F(M)) = \left.\left\{\left(z_1,
        \frac{z_2+2}{z_1 ^2 + h(z_2)}\right)\, \right|\, (z_1,z_2) \in
    X \right\}.
\end{align*}
Note that the second component of $g$ is an even function of $z_1$ and
hence the range $\widetilde{F}(M)$ is symmetric about the vertical
axis in $\mathbb{R}^2$. A straightforward analysis shows that
$\widetilde{F}(M)$ is the following region in $\mathbb{R}^2$:
\[
\left\{(x,y) \in \mathbb{R}^2 \,\left|\; 0<|x|\leq 1,\, \frac{1}{x^2 +
      h(-1)}\leq y \leq \frac{3}{x^2} \right\}\right.  \bigcup
\left(\{0\} \times \left[\frac{1}{h(-1)}, \infty\right) \right);
\]
see Figure \ref{Fig_im_tilde_F}.
\begin{figure}[h]
  \centering
  \includegraphics[height=6.5cm]{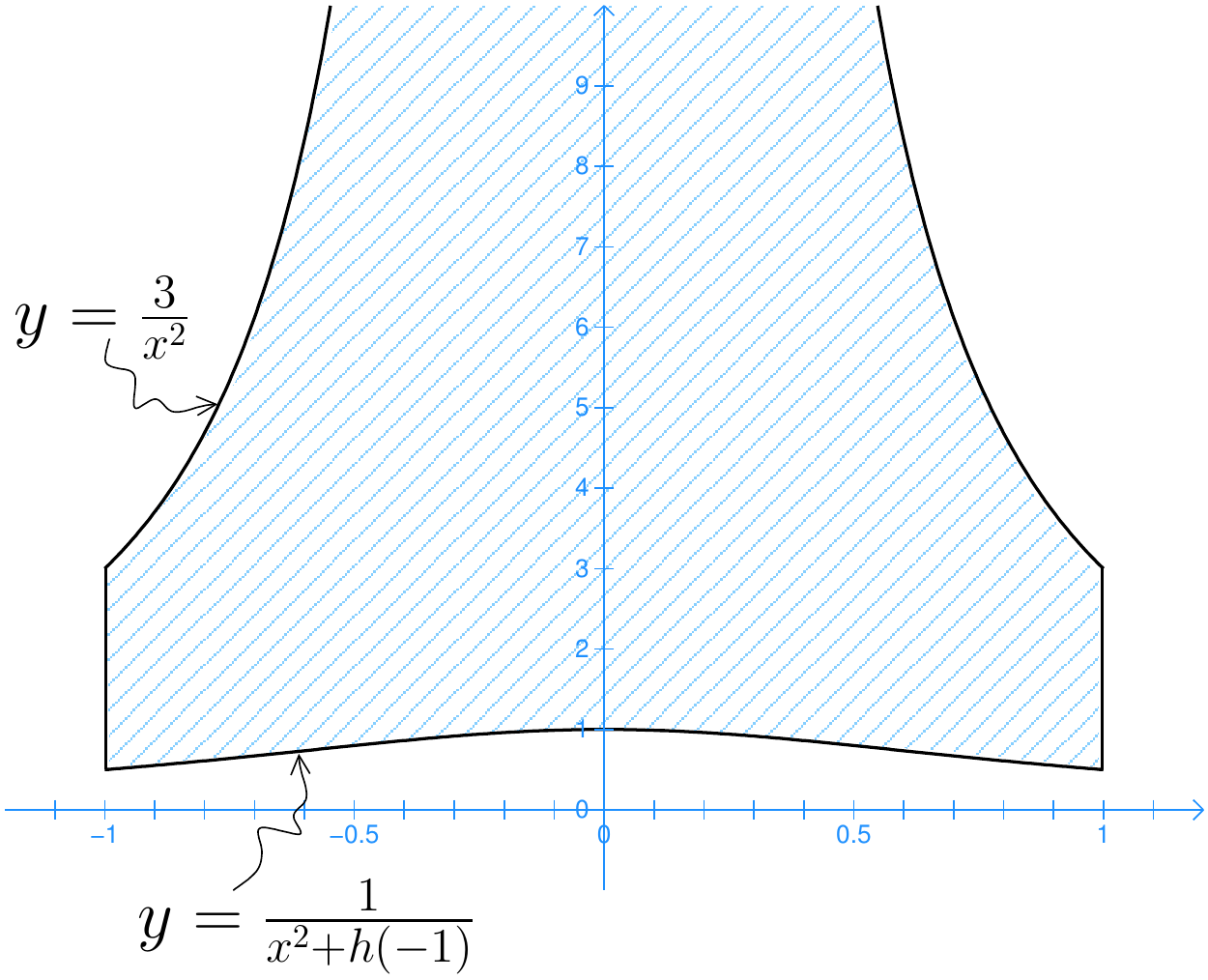}
  \caption{The set $\widetilde{F}(M)$ with the choice $h(-1)=1$.  }
  \label{Fig_im_tilde_F}
\end{figure}

Note that the closed segment $[-1,1] \times \{-1\} \subset F(M)$ is
mapped by $g$ to the lower curve in Figure \ref{Fig_im_tilde_F}, the
two half-open segments $([-1,1]\setminus\{0\}) \times \{1\}$ to the
two upper curves, the two closed vertical segments to the two closed
vertical segments, and the half-open interval $\{0\} \times [-1, 0)$
to the infinite half-open interval $\{0\} \times [1/h(-1), \infty)$.

\paragraph{{\bf Step 9}} (\textit{Construction of the cartographic
  representation.})  We shall construct the cartographic invariant in Theorem
\ref{theo:b} from $\widetilde{F}(M)$ by flattening out the
horizontal curves and setting the height between them at the value
given by the volume of the corresponding reduced phase space.  For
each $|x|\leq 1$, let $\ell(x)$ denote the volume of the reduced
manifold $J^{-1}(x)/S^1$. Then, by Theorem \ref{theo:b}, the
cartographic invariant associated to the general semitoric system $(M,\widetilde{F})$
is given by the formula
$$\Delta = \left\{(x,y) \in\mathbb{R}^2 
  \mid 0<|x| \leq 1, \, 0 \leq y \leq \ell(x) \right\} \cup \{0\}
\times [0, 2 \pi).$$ Using the definition \eqref{function_f} of $f$, a
direct computation shows that if $x<0$ then $J^{-1}(x)/S^1 = \{x\}
\times S^2$, and hence $$\ell(x) = \int_{S^2} f(x,z_2) {\rm d}\theta
\wedge {\rm d}z_2 = 2 \pi\int_{-1}^1 f(x,z_2) {\rm d}z_2 = 4\pi,$$
because for $x<0$, we have $f(x,z_2) =1$ for any $z_2 \in [-1,1]$.
Similarly, if $x>0$ then, as before, the reduced space is
$J^{-1}(x)/S^1= \{x\} \times S^2$, and hence $\ell(x) = 2\pi\int_1^1
\chi(z_2) {\rm d}z_2$.  If $x=0$, then the reduced space
$J^{-1}(0)/S^1$ is the southern hemisphere of the second factor and
hence $\ell(0) = 2 \pi$.  Therefore, the cartographic invariant is given in Figure
\ref{polygon1}.
\begin{figure}[h]
  \centering
  \includegraphics[height=0.3\textwidth] {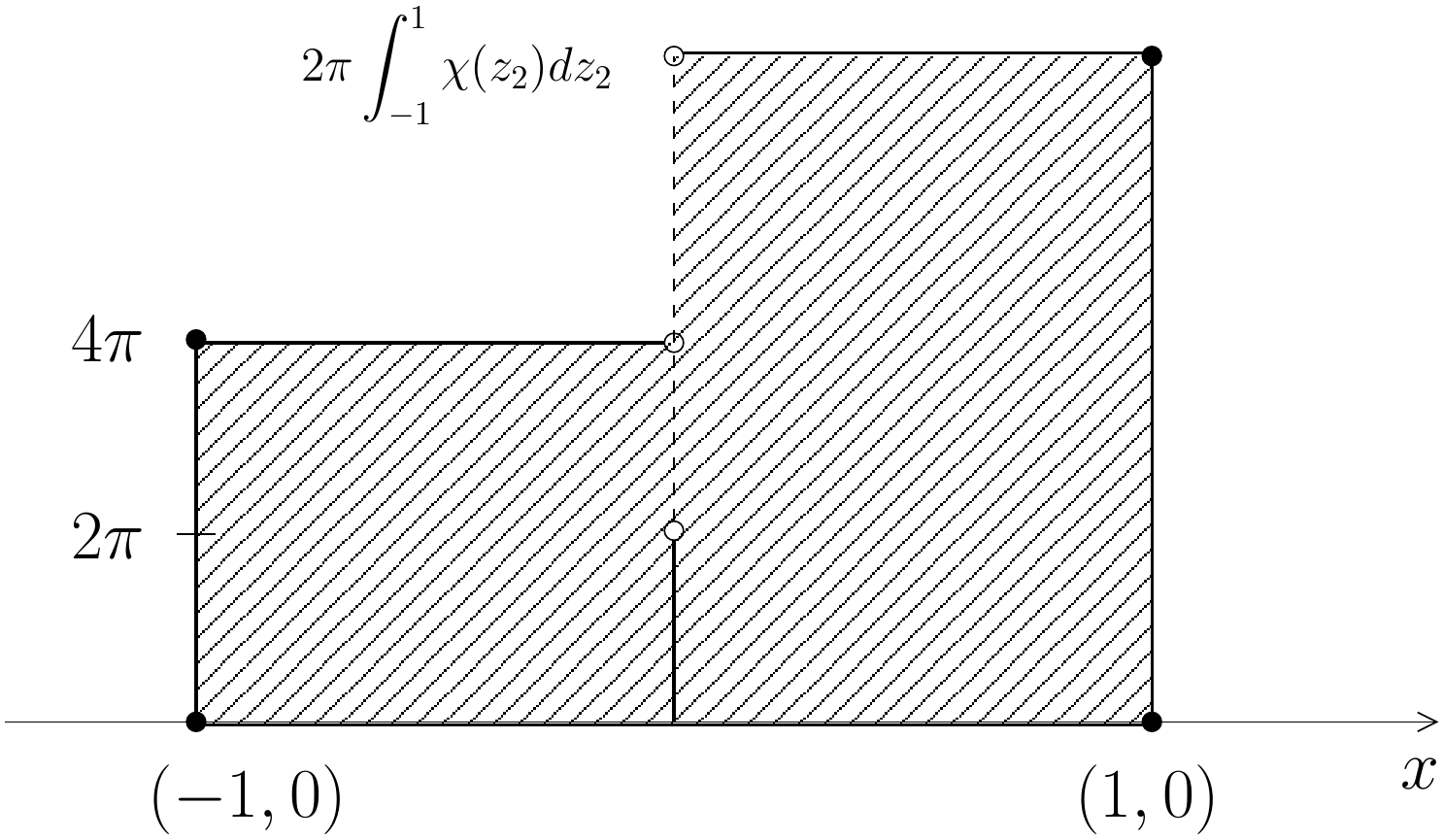}
  \caption{A representative of $\Delta(M, \widetilde{F})$.}
  \label{polygon1}
\end{figure}

We have so far shown (E.\ref{d1})\--(E.\ref{d5}). Theorem~\ref{theo:c}
implies (E.\ref{d6}). We have left to show (E.\ref{d7}).

\medskip

To conclude the proof, we modify the construction above in order to
illustrate the existence of unbounded cartographic invariants with fibers of infinite
length. As we shall see, most of the computations of the previous
example remain valid. Let
\begin{align}
  \label{N_manifold}
  N:=S^2 \times S^2 \setminus &\Big( \{((x_1,y_1,z_1),\,(x_2,y_2,z_2))
  \in S^2 \times S^2 \mid z_1=0,\, z_2
  \geq 0\}  \nonumber \\
  & \qquad \cup \{((x_1,y_1,z_1),\, (x_2,y_2,z_2)) \in S^2 \times S^2
  \mid z_1 \geq 0,\,z_2=1\}\Big).
\end{align}
As in the previous example, $N$ is open and connected. Moreover,
because it is a subset of $M$, the restriction of the form $\Omega$
given by \eqref{ex_omega}, is a symplectic form. Similarly, $J=z_1$,
$H=z_2$ defines an integrable system on $N$ and $J$ is the momentum
map of a Hamiltonian $S^1$-action. The computations in the previous
example show that we have the same singularities, all of them
non-degenerate.  If $F=(J,H)$, its image is
\begin{equation}
  \label{image_of_F_modified}
  F(N)=[-1,\,1]\times [-1,\,1] 
  \setminus \big(\{ z_1 =0, \, z_2\ge 0\} 
  \cup \{0\leq z_1\leq 1,\, z_2=1\}\big)
\end{equation} 
(see Figure \ref{Fig_im_F_modified}) which is not a closed set, and
hence $F$ is not a proper map.
\begin{figure}[h]
  \centering
  \includegraphics[height=4.5cm]{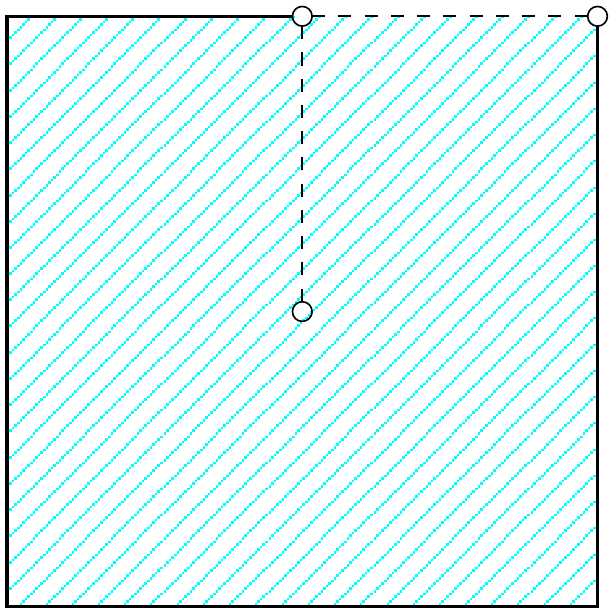}
  \caption{The image $F(N)$.}
  \label{Fig_im_F_modified}
\end{figure}
Define
\[
g(z_1,\, z_2):=\left(z_1,\, \frac{z_2+2}{((z_1-1)^2 +
    h(z_1))(z_1^2+h(z_2))}\right)
\]
and $\widetilde{F}: = g \circ F$, where $F:=(z_1, z_2)$; $h$ is as in
the previous example. To see that $g$ is a local diffeomorphism, it
suffices to note that the Jacobian determinant of $g$ has the
expression $\big(\Delta - (z_2+2) \frac{\partial\Delta}{\partial
  z_2}\big)/ \Delta^2$, where $\Delta:=((z_1-1)^2 +
h(z_1))(z_1^2+h(z_2))$.  Since $\Delta>0$
and $$\partial\Delta/\partial z_2 = 2(z_2-1)(z_1 ^2+h(z_2)) +
((z_1-1)^2 + h(z_1))h'(z_2) <0,$$ it follows that the Jacobian
determinant of $g$ is strictly positive. As in the previous example,
one can check that $g^{-1}(K_1 \times K_2)$ is a compact subset of
$\mathbb{R}^2$, where $K_i$, $i=1,2$, are closed bounded intervals in
$\mathbb{R}$. The argument given in the previous example shows then
that $\widetilde{F}$ is a proper map. Therefore, $(N, \widetilde{F})$
is a proper general semitoric system.  The image $\widetilde{F}$ is
given in Figure \ref{boot}.
\begin{figure}[h]
  \centering
  \includegraphics[height=0.7\textwidth]{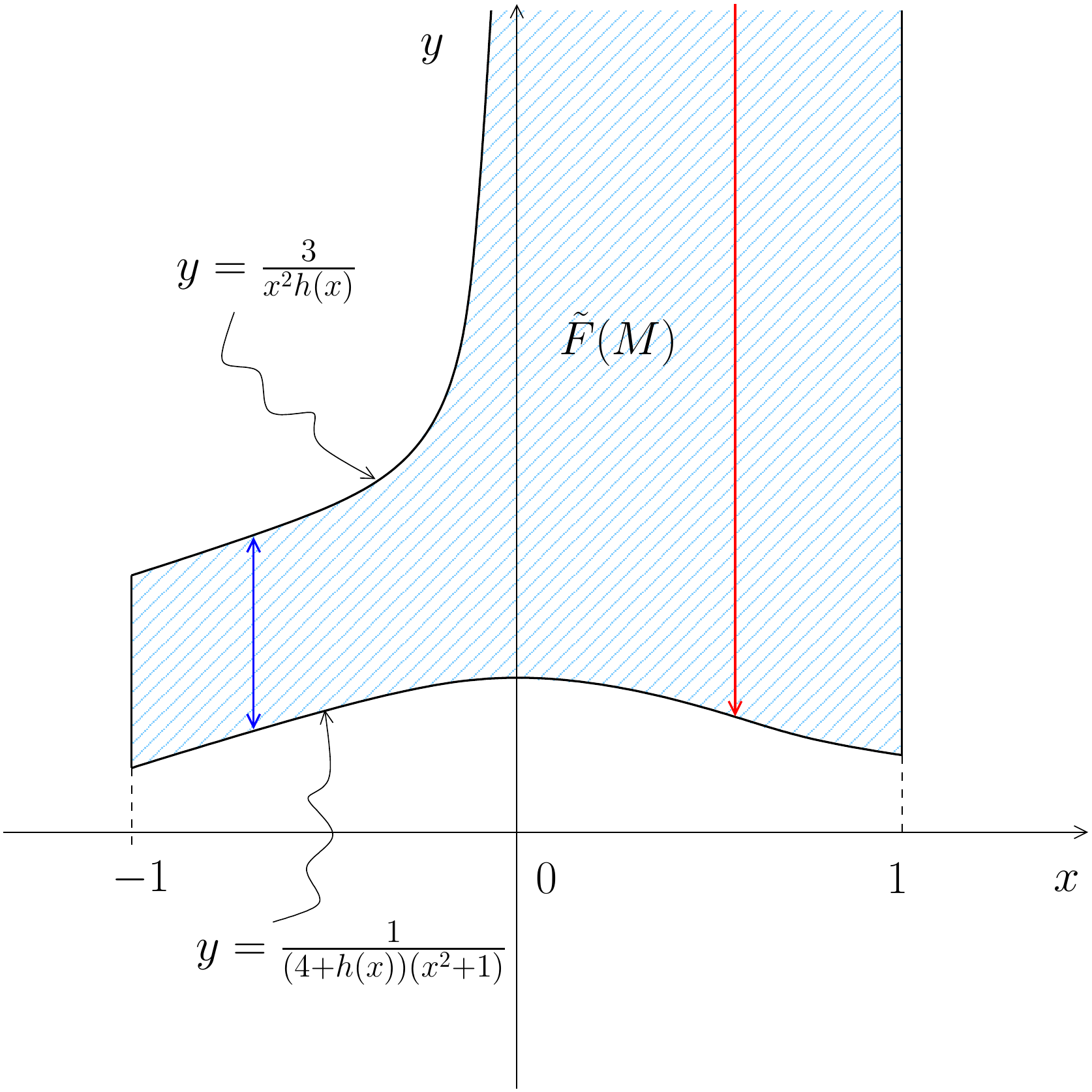}
  \caption{The image $\widetilde{F}(N)$ with the choice $h(-1)=1$.}
  \label{boot}
\end{figure}

Finally, to determine the possible affine invariants associated to this system,
we need to compute $\ell(x)$, the volume of the reduced manifold
$J^{-1}(x)/S^1$. As before, we compute
\[
\ell(x) = \left\{
  \begin{aligned}
    4 \pi, \quad \text{if} \quad x<0\\
    2 \pi, \quad \text{if} \quad x=0\\
    2 \pi(1 + \alpha), \quad \text{if} \quad x>0
  \end{aligned}
\right.
\]
where $\alpha: = \int_0^1 \chi(z_2) {\rm d}z_2\geq 0$. The possible
cartographic invariants are given in Figure \ref{bootpolygon}.
\begin{figure}[h]
  \centering
  \includegraphics[width=0.7\textwidth]{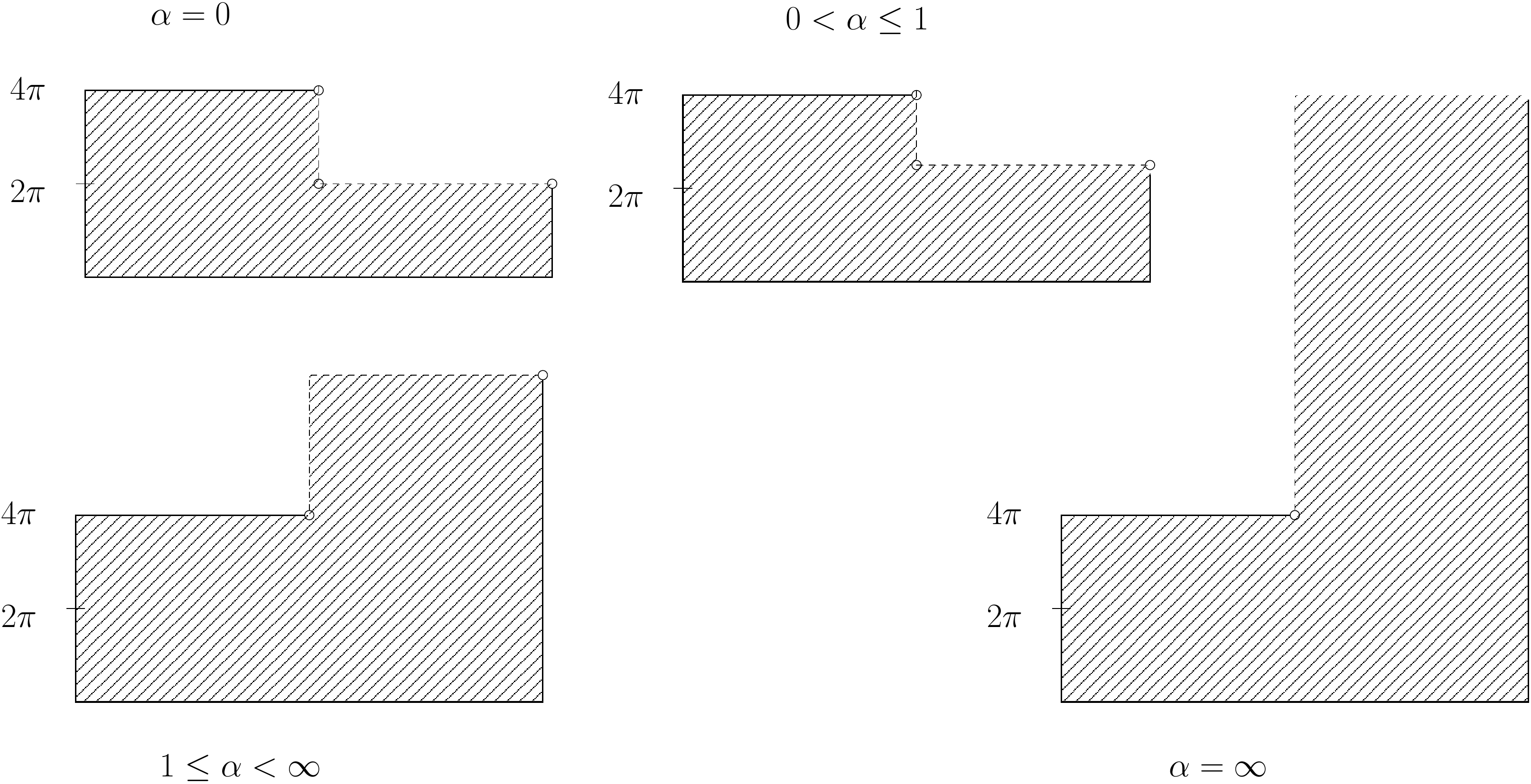}
  \caption{A representative of the cartographic invariants depending on
    $\alpha$.}
  \label{bootpolygon}
\end{figure}
This proves (E.\ref{d7}).
 
\begin{remark}
  When $M$ is compact (so $J,H,F$ are all proper), the cartographic invariant of $F$ is  a polygon which
  is related to the classification of Hamiltonian $S^1$\--spaces by Karshon
  \cite{Ka99}, as explained in \cite{HSS13}.
\end{remark}

\section{Appendix}

\subsection{Bifurcation set}

Let $M$ and $N$ be smooth manifolds. A smooth map $f: M\rightarrow N$
is said to be \emph{locally trivial} at $n_0 \in f(M)$, if there is an
open neighborhood $U \subset N$ of $n_0$ such that $f^{-1}(n)$ is a
smooth submanifold of $M$ for each $n\in U$ and there is a smooth map
$h: f^{-1}(U)\rightarrow f^{-1}(n_0)$ such that $f\times
h:f^{-1}(U)\rightarrow U \times f ^{-1}(n_0)$ is a diffeomorphism. The
\emph{bifurcation set} $\Sigma_f$ consists of all the points of $N$
where $f$ is not locally trivial.

It is known that the set of critical values of $f$ is included in the
bifurcation set and that if $f$ is proper this inclusion is an
equality (see \cite[Proposition 4.5.1]{AbMa1978} and the comments
following it).

\subsection{Linearization of singularities}
\label{sec:Eliasson}

Let $(M, \omega)$ be a connected symplectic $4$-manifold,
$F=(f_1,f_2)$ an integrable system on $(M, \omega)$, and $m \in M $ a
critical point of $F$, i.e., the rank of the derivative (tangent map)
${\rm d}_m F: T_mM \rightarrow \mathbb{R}^2$ of $F$ is either $0$ or
$1$. If ${\rm d}_mF=0$, $m$ is said to be \emph{non-degenerate} if the
Hessians $\operatorname{Hess}f_1(m)$, $\operatorname{Hess}f_2(m)$ span
a Cartan subalgebra of the symplectic Lie algebra of quadratic forms
on the symplectic vector space $({\rm T}_m M, \omega_m)$. If
$\operatorname{rank}\left({\rm d}_mF\right)=1$, we may assume that
${\rm d}_m f_1 \neq 0$.  Let $\iota \colon S \to M$ be an embedded
local $2$-dimensional symplectic submanifold through $m$ such that
${\rm T}_m S\subset \ker({\rm d}_m f_1)$ and ${\rm T}_mS$ is
transversal to the Hamiltonian vector field $\mathcal{X}_{f_1}$
defined by the function $f_1$. This is possible by the classical
Hamiltonian Flow Box Theorem (\cite[Theorem 5.2.19]{AbMa1978}), also
known as the Darboux-Carath\'eodory Theorem (\cite[Theorem
4.1]{PeVN2011}). It is easily seen that the definition does not depend
on the choice of $S$. The point $m$ is called \textit{transversally
  non-degenerate} if $\operatorname{Hess}(\iota^\ast f_2)(m)$ is a
non-degenerate symmetric bilinear form on ${\rm T}_m S$.

For the notion of non\--degeneracy of a critical point in arbitrary
dimensions see \cite{vey} and \cite[Section 3]{san-mono}). In this
paper, we need the following property of non-degenerate critical
points (\cite{Eliasson1984, Eliasson1990}, \cite{VNWa2010}) in terms
of the Williamson normal form (\cite{Williamson1936}), which we state
in any dimension but will only use in dimension $4$.

\begin{theorem}[Eliasson]
  \label{singularities_theorem}
  Let $F=(f_1,\,\ldots,f_n): M \rightarrow \mathbb{R}^n$ be an
  integrable system and $m \in M$ a non-degenerate critical point of
  $F$.  Then there are local symplectic coordinates $(x_1,\,
  \ldots,x_n,\, \xi_1,\,\ldots,\, \xi_n)$ about $m$, in which $m$ is
  represented as $(0,\,\ldots,\, 0)$, such that $\{f_i,\,q_j\}=0$, for
  all $i,\,j$, where the $q_1,\,\ldots,\,q_n$ are defined on a
  neighborhood of $(0,\,\ldots,\,0)$ in $\mathbb{R}^n$ and have one of
  the following expressions:
  \begin{itemize}
  \item[{\rm (a)}] Elliptic component: $q_j = (x_j^2 + \xi_j^2)/2$,
    where $1 \le j \le n$.
  \item[{\rm (b)}] Hyperbolic component: $q_j = x_j \xi_j$, where $1
    \le j \le n$.
  \item[{\rm (c)}] Focus\--focus component: $q_{j-1}=x_{j-1}\, \xi_{j}
    - x_{j}\, \xi_{j-1}$ and $q_{j} =x_{j-1}\, \xi_{j-1} +x_{j}\,
    \xi_{j}$ where $2 \le j \le n-1$.
  \item[{\rm (d)}] Non\--singular component: $q_{j} = \xi_{j}$, where
    $1 \le j \le n$.
  \end{itemize}
  If $m$ does not have hyperbolic components, then the system of
  equations $\{f_i,\,q_j\}=0$, for all $i,\,j$, may be replaced by $
  (F-F(m))\circ \varphi = g \circ (q_1,\, \ldots,\,q_n), $ where
  $\varphi=(x_1,\, \ldots,x_n,\, \xi_1,\, \ldots,\, \xi_n)^{-1}$ and
  $g$ is a diffeomorphism from a neighborhood of $(0,\,\ldots,\,0)$ in
  $\R^n$ onto another such neighborhood, with $g(0,\,\ldots,\,0)=
  (0,\,\ldots,\,0)$.
\end{theorem}

If $M$ is $4$\--dimensional and $F$ has no hyperbolic singularities,
$(q_1,\,q_2)$ is:
\begin{enumerate}[{\rm (T.1)}]
\item \label{t1} if $m$ is a critical point of $F$ of rank zero, then
  $q_j$ is one of
  \begin{itemize}
  \item[{\rm (i)}] $q_1 = (x_1^2 + \xi_1^2)/2$ and $q_2 = (x_2^2 +
    \xi_2^2)/2$.
  \item[{\rm (ii)}] $q_1=x_1\xi_2 - x_2\xi_1$ and $q_2
    =x_1\xi_1+x_2\xi_2$;
  \end{itemize}
  on the other hand,
\item \label{t2} if $m$ is a critical point of $F$ of rank one, then
  \begin{itemize}
  \item[{\rm (iii)}] $q_1 = (x_1^2 + \xi_1^2)/2$ and $q_2 = \xi_2$.
  \end{itemize}
\end{enumerate}
A non-degenerate critical point is called \textit{elliptic-elliptic,
  focus-focus}, or \textit{trans\-ver\-sally-elliptic} if both
components $q_1,\, q_2$ are of elliptic type, $q_1,\,q_2$ together
correspond to a focus-focus component, or one component is of elliptic
type and the other component is $\xi_1$ or $\xi_2$, respectively.

For the spherical pendulum, see Figure
\ref{critical_set_spherical_pendulum.figure} where the critical points
of $F$ lie in $F({\rm T}S^2)$.

\subsection{Affine manifolds} \label{sec:affine}

An \textit{affine $n$-dimensional manifold} is a smooth manifold
admitting an atlas whose change of chart maps are in the affine group
of $\mathbb{R}^n$, i.e., in
\begin{align*}
  \operatorname{Aff}(n, \mathbb{R})&:=
  \operatorname{GL}(n, \mathbb{R}) \ltimes \mathbb{R}^n\\
  &: = \left.\left\{
      \begin{bmatrix}
        {U}&{u}\\
        {0}&1
      \end{bmatrix} \;\right| \; {U} \in \operatorname{GL}(n,
    \mathbb{R}), \;\; u \in \mathbb{R}^2 \right\}\subset
  \operatorname{GL}(n+1, \mathbb{R}).
\end{align*}
An \textit{integral affine $n$-dimensional manifold} is an affine
manifold admitting an atlas whose change of chart maps are in
$\operatorname{Aff}(n, \mathbb{Z}):= \operatorname{GL}(n, \mathbb{Z})
\ltimes \mathbb{R}^n$, i.e., $U \in \operatorname{GL}(n, \mathbb{Z})$
in the definition above.

Let $M$ be a connected $n$-dimensional manifold, $m_0 \in M$, and $p:
\widetilde{M} \rightarrow M$ its universal covering manifold, i.e.,
the set of homotopy classes of smooth paths $\lambda:[0,1] \rightarrow
M$ starting at $\lambda(0) = m_0$ and keeping the endpoints fixed;
$p([\lambda]) :=\lambda(1)$.  Recall that $\widetilde{M}$ is a smooth
simply connected $n$-dimensional manifold and that $p$ is a covering
map.  The group of deck transformations of $p$, i.e., all
diffeomorphisms $\chi:\widetilde{M}\rightarrow \widetilde{M}$ such
that $p \circ \chi = p$, is isomorphic to the first fundamental group
$\pi_1(M)$ (based at $m_0$).

If $M$ is, in addition, an affine manifold (see, e.g., \cite[Section
2.3]{GoHi1984} for more information), then $p$ induces an affine
manifold structure on $\widetilde{M}$ by requiring $p$ to be an affine
map, i.e., its local representative is affine in any pair of local
charts. A \textit{developing map} for $M$ is an affine immersion
$\zeta:\widetilde{M} \rightarrow \mathbb{R}^n$. It is well-known (see,
e.g. \cite[page 641]{GoHi1984}) that each connected affine manifold
has at least one developing map and that if $\zeta': \widetilde{M}
\rightarrow \mathbb{R}^n$ is another developing map then there is a
unique $A \in \operatorname{Aff}(n, \mathbb{R})$ such that $\zeta' =
{A}\zeta$. In addition, for any developing map $\zeta: \widetilde{M}
\rightarrow \mathbb{R}^n$, there is a unique equivariant
\textit{monodromy homomorphism} $\mu: \pi_1(M) \rightarrow
\operatorname{Aff}(n,\mathbb{R})$, i.e., $\zeta([\lambda\star \gamma])
= \mu([\lambda]) \zeta([\gamma])$ for any $[\lambda] \in \pi_1(M)$ and
$[\gamma] \in \widetilde{M}$, where $\star$ denotes composition of
paths by concatenation.

\medskip

\paragraph{\textbf{Acknowldegments.}} 
This paper was completed under the excellent hospitality of the Bernoulli Center
in Lausanne, where the first and third authors are co\--organizers of
a program with Nicolai Reshetikhin, on semiclassical analysis and
integrable systems. AP was partially supported by NSF CAREER Grant
DMS-1055897 and NSF Grant DMS-0635607 and an Oberwolfach Leibniz
Fellowship.  TSR was partially supported by the government grant of
the Russian Federation for support of research projects implemented by
leading scientists, Lomonosov Moscow State University under the
agreement No.  11.G34.31.0054 and Swiss NSF grant 200021-140238. SVN
was partially supported by Institut Universitaire de France, the
Lebesgue Center (ANR Labex LEBESGUE), and the ANR NOSEVOL grant.

\noindent
\\
\noindent
\\
{{\'A}lvaro Pelayo}\\
School of Mathematics\\
Institute for Advanced Study\\
One Einstein Drive\\
Princeton, NJ 08540, USA.
\\
\\
Washington University\\
Mathematics Department\\
One Brookings Drive, Campus Box 1146\\
St Louis, MO 63130-4899, USA.
\\
{\em E\--mail}: {apelayo@math.wustl.edu}\\
\url{http://www.math.wustl.edu/~apelayo}

\medskip\noindent

\smallskip\noindent
Tudor S. Ratiu\\
Section de Math\'ematiques and Bernoulli Center\\
Station 8\\
Ecole Polytechnique F\'ed\'erale
de Lausanne\\
CH-1015 Lausanne, Switzerland\\
{\em E\--mail}: \texttt{tudor.ratiu@epfl.ch}

\medskip\noindent

\medskip\noindent

\noindent
\noindent
San V\~u Ng\d oc \\
Institut Universitaire de France
\\
\\
Institut de Recherches Math\'ematiques de Rennes\\
Universit\'e de Rennes 1\\
Campus de Beaulieu\\
F-35042 Rennes cedex, France\\
{\em E-mail:} \texttt{san.vu-ngoc@univ-rennes1.fr}\\
{\em Website}: \url{http://blogperso.univ-rennes1.fr/san.vu-ngoc/}

\end{document}